\documentclass[10pt,reqno]{amsart}

\usepackage[utf8]{inputenc}
\usepackage[english]{babel}
\usepackage[dvipsnames]{xcolor}

\usepackage{hyperref}
\usepackage{xparse}
\NewDocumentCommand{\ceil}{s O{} m}{%
  \IfBooleanTF{#1} 
    {\left\lceil#3\right\rceil} 
    {#2\lceil#3#2\rceil} 
}

\usepackage{enumerate}

\newtheorem{theorem}{Theorem}[section]
\newtheorem{lemma}[theorem]{Lemma}

\theoremstyle{remark}
\newtheorem{remark}[theorem]{\bf{Remark}}

\theoremstyle{definition}

\newtheorem{definition}[theorem]{Definition}
\newtheorem{proposition}[theorem]{Proposition}

\newcommand\cbrk{\text{$]$\kern-.15em$]$}}
\newcommand\opar{\text{\,\raise.2ex\hbox{${\scriptstyle
|}$}\kern-.34em$($}}
\newcommand\cpar{\text{$)$\kern-.34em\raise.2ex\hbox{${\scriptstyle |}$}}\,}

\newcommand{\aint}{-\hspace{-0.38cm}\int}

\newcommand\bL{\mathbb{L}}

\newcommand\bR{\mathbb{R}}
\newcommand\bH{\mathbb{H}}
\newcommand\bZ{\mathbb{Z}}

\newcommand\bN{\mathbb{N}}
\newcommand\bM{\mathbb{M}}

\newcommand\cO{\mathcal{O}}

\newcommand\ep{\varepsilon}

\newcommand{\mysection}[1]{\section{#1}
\setcounter{equation}{0}}

\newcommand{\Ccinf}{C_{c}^{\infty}}
\newcommand{\R}{\mathbb{R}}

\begin{document}

\title[weighted $L_q(L_p)$-estimates  for time-fractional equations]
{Weighted $L_q(L_p)$-estimate with Muckenhoupt weights for the diffusion-wave equations with time-fractional derivatives}

\author{Beom-seok Han}
\address{Department of Mathematics, Korea University, 1 Anam-dong, Sungbuk-gu, Seoul,
136-701, Republic of Korea} \email{hanbeom@korea.ac.kr}

\author{Kyeong-Hun Kim}
\address{Department of Mathematics, Korea University, 1 Anam-dong,
Sungbuk-gu, Seoul, 136-701, Republic of Korea}
\email{kyeonghun@korea.ac.kr}
\thanks{}

\author{Daehan Park}
\address{Department of Mathematics, Korea University, 1 Anam-dong, Sungbuk-gu, Seoul,
136-701, Republic of Korea} \email{daehanpark@korea.ac.kr}



\begin{abstract}


We present  a weighted $L_{q}(L_{p})$-theory ($p,q\in(1,\infty)$) with Muckenhoupt weights for the equation 
$$
\partial_{t}^{\alpha}u(t,x)=\Delta u(t,x) +f(t,x), \quad t>0, x\in \bR^d.
$$
Here, $\alpha\in (0,2)$ and $\partial_{t}^{\alpha}$ is the Caputo fractional derivative of order $\alpha$. In particular we prove that for any $p,q\in (1,\infty)$, $w_{1}(x)\in A_p$ and $w_{2}(t)\in A_q$,
$$
\int^{\infty}_0\left(\int_{\bR^d} |u_{xx}|^p \,w_{1} dx \right)^{q/p}\,w_{2}dt \leq N \int^{\infty}_0\left(\int_{\bR^d} |f|^p \,w_{1} dx \right)^{q/p}\,w_{2}dt,
$$
where $A_p$ is the class of   Muckenhoupt $A_p$  weights. Our approach is based on the sharp function estimates of the derivatives of solutions.

\end{abstract}

\maketitle

\mysection{Introduction}

Let $\alpha \in (0,2)$ and $\partial^{\alpha}_t$ denote the Caputo derivative of order $\alpha$. The  equation
\begin{equation} \label{main_equation_non_div}
\partial_t^\alpha u(t,x) = \Delta u(t,x) +f(t,x),\quad t>0\, ;\, u(0,\cdot)=1_{\alpha>1}\frac{\partial u}{\partial t}(0,\cdot)=0
\end{equation}
 describes different phenomena according to the range of $\alpha$.  The heat equation ($\alpha=1$) represents the heat propagation in homogeneous media. For $\alpha \in (0,1)$, the  equation describes subdiffusive aspect of the anomalous diffusion, caused by particle sticking and trapping effects (see e.g. \cite{metzler1999anomalous,metzler2000random}). If $\alpha\in(1,2)$, the fractional wave equation   gives information of wave propagating in viscoelastic media (see e.g. \cite{mainardi1995fractional,mainardi2001fractional}). 

In this article, we prove the unique solvability of equation \eqref{main_equation_non_div} with zero initial data in a weighted $L_q(L_p)$-spaces.  In particular, we prove that for any $p,q\in (1,\infty)$, $w_{1}=w_{1}(x)\in A_p$ and $w_{2}=w_{2}(t)\in A_q$, it holds that  

$$
\| |\partial^{\alpha}_{t}u|+|u|+|u_{x}|+|u_{xx}| \|_{\bL_{q,p}(w_{2},w_{1},T)} \leq N \| f \|_{\bL_{q,p}(w_{2},w_{1},T)},
$$
where the norm in $\bL_{q,p}(w_{2},w_{1},T)$ is defined by
$$
\| f \|_{\bL_{q,p}(w_{2},w_{1},T)} := \left( \int_0^T \left( \int_{\bR^d} |f(t,x)|^p w_{1}(x)dx\right)^{q/p}w_{2}(t)dt\right)^{1/q}.
$$
Here $A_p$  denotes the class of Muckenhoupt $A_p$ weights.  See Definition \ref{def 05.06.1} for the definition. 

 Here is a  short description on the closely related works (that is $L_q(L_p)$-theory) and comparison to our result.   
A standard or unweighted $L_q(L_p)$-estimate for the equation (and for more general Volterra equations)
\begin{equation}
  \label{eqn 10.20}
\partial^{\alpha}_tu=a^{ij}(t,x)u_{x^ix^j}(t,x)+f
\end{equation}
was introduced in \cite{clement1992global, Pr1991} under the conditions  $a^{ij}=\delta^{ij}$, $\alpha \in (0,1)$, and 
$$
\frac{2}{\alpha q}+\frac{d}{p}<1.
$$
The results of \cite{clement1992global, Pr1991} are based on operator theory (or semigroup theory), and similar approach is used  in \cite{zacher2005maximal} for general $a^{ij}(t,x)$  under the conditions that $p=q>1$, $a^{ij}$ are uniformly continuous  in $(t,x)$, and 
$$
\alpha\not\in \{\frac{2}{2p-1}, \frac{2}{p-1}-1,\frac{1}{p}, \frac{3}{2p-1}\}.
$$
In \cite{kim17timefractionalpde}, under the continuity condition of $a^{ij}(t,x)$, above restrictions on $\alpha$ and $p,q$ are dropped and it is only assumed that $\alpha\in (0,2)$ and $p,q>1$. The   Calder\'on-Zygmund theorem is mainly used in  \cite{kim17timefractionalpde}.  
 Quite recently, using level set arguments,  conditions on $a^{ij}$ are significantly relaxed in \cite{dong2019lp}. More precisely,  it is only assumed that the coefficients are measurable in $t$ and have small mean oscillation in $x$. However the approach of \cite{dong2019lp} only covers the case $p=q$ and $\alpha\in (0,1)$, and all the above-mentioned results are handled in Sobolev spaces without weights.

Our result substantially generalizes the previous results in the sense that  we cover general $L_q(L_p)$-theory with Muckenhoupt weights  and we do not impose any algebraic conditions on $\alpha$, $p$, and $q$. This is possible since we use a different approach. We control the sharp functions of solutions and their derivatives in terms of maximal functions of free terms, and apply Fefferman-Stein and Hardy-Littlewood theorems to obtain a priori estimates. Such approach is a typical tool in the theory of PDEs with local operators,   but  has not been  used well (if any)  for equation \eqref{eqn 10.20} mainly because the sharp function estimates are based on local estimates of solutions which are non-trivial  for equations with non-local operators.

We also remark that in this article we only cover the case $a^{ij}=\delta^{ij}$ because our estimations depend on  upper bounds of kernel appearing in the representation of solutions of equation \eqref{main_equation_non_div}. Obviously our results hold if $a^{ij}$ are constants, and moreover our approach works for   equation \eqref{eqn 10.20} with variable coefficients   if one can obtain sharp upper bounds of derivatives of the kernel  related to  the equation.

This article is organized as follows. In Section 2, we introduce some definitions and facts related to  fractional calculus, and we present our main result, Theorem 
\ref {theorem 5.1}.  In Section 3, we prove  a priori estimates of solutions to \eqref{main_equation_non_div}. In Section 4, we prove the main theorem.

We finish the introduction with  notation used in this article. $\mathbb{N}$ stands for the set of positive integers. $\mathbb{R}^d$ denotes the $d$-dimensional Euclidean space of points $x = (x^1,\dots,x^d)$. $B_r(x) = \{ y\in\mathbb{R}^d:|x-y|<r \}$ and $B_r = B_r(0)$. If a set $E$ is in $\R^{d}$ (or $\R^{d+1}$), then $|E|$ is the Lebesgue measure of $E$. For $i=1,...,d$, multi-indices $\gamma=(\gamma_{1},...,\gamma_{d})$,
$\gamma_{i}\in\{0,1,2,...\}$, and functions $u(t, x)$ we set
$$
u_{x^{i}}=\frac{\partial u}{\partial x^{i}}=D_{i}u,\quad
D^{\gamma}u=D^{\gamma}_xu=D_{1}^{\gamma_{1}}\cdot...\cdot D^{\gamma_{d}}_{d}u.
$$
We also use the notation $D^m$ (or $D^m_x$) for  partial derivatives of order $m$ with respect to $x$.  
 Similarly,
by $\partial_{t}^nu$ (or $\frac{d^n}{dt^n}u$) we 
mean a partial derivative  of order $n$   with respect to $t$. $C_c^\infty(\cO)$ denotes the collection of all infinitely differentiable functions with compact support in $\cO$, where $\cO$ is an open set in $\mathbb{R}^d$ or $\R^{d+1}$. For a measure space $(E,\mu)$, a Banach space $B$, and $p\in(1,\infty)$, $L_p(E,\mu;B)$ denotes the set of $B$-valued $\mu$-measurable functions $u$ on $E$ satisfying
\begin{equation*}
\| u \|_{L_p(E,\mu;B)} := \left(\int_E\|u\|_B^p d\mu\right)^{1/p}<\infty.
\end{equation*}
If $B = \R$, then $L_p(E,\mu;B) = L_p(E,\mu)$. For a measurble set $A$ and a measurable function $f$, we use the following notation
\begin{equation*}
\aint_{A}f(x)d\mu :=\frac{1}{\mu(A)}\int_{A}f(x)d\mu.
\end{equation*} 
Finally if we write $N=N(a,b,\ldots)$, then this means that the constant $N$ depends only on $a,b,\ldots$.

\mysection{Main result}

We first introduce some definitions and facts related to the fractional calculus. For more details, see e.g. \cite{baleanu12fractional,podlubny98fractional,richard14fractional,samko93fractional}.
For $\alpha>0$ and $\varphi\in L_1((0,T))$, the Riemann-Liouville fractional integral  of order $\alpha$ is defined by
$$ I^\alpha_{t} \varphi(t) :=(I^\alpha_{t} \varphi) (t):= \frac{1}{\Gamma(\alpha)}\int_0^t(t-s)^{\alpha-1}\varphi(s)ds,\quad  t\leq T.
$$
By H\"older's inequality, for any $p \geq 1$, 
\begin{equation*}
\|I^\alpha_t \varphi \|_{L_p((0,T))}\leq N(\alpha,p,T)\| \varphi \|_{L_p((0,T))}.
\end{equation*}
It is also easy to check that  if $\varphi$ is bounded then $I^{\alpha}_t \varphi(t)$ is a continuous function satisfying $I^{\alpha}_t\varphi(0)=0$.   

Let $n$ be the integer such that $n-1\leq \alpha<n$. If $\varphi$ is $(n-1)$-times differentiable, and $(\frac{d}{dt})^{n-1}I_{t}^{n-\alpha}\varphi$ is absolutely continuous on $[0,T]$, then the Riemann-Liouville fractional derivative $D_t^\alpha \varphi$ and  the Caputo fractional derivative $\partial_t^\alpha\varphi$ are defined as follows.
\begin{equation*}
D^{\alpha}_t \varphi (t):=(D^{\alpha}_t \varphi) (t):=(I_{t}^{n-\alpha}\varphi)^{(n)}(t),
\end{equation*}
\begin{equation}\label{Caputo_Riemann}
\partial^{\alpha}_{t}\varphi (t) :=D^{\alpha}_{t}\left(\varphi(t)-\sum_{k=0}^{n-1}\frac{t^{k}}{k !}\varphi^{(k)}(0)\right)(t).
\end{equation}
Obviously, $D^\alpha_t\varphi = \partial_t^\alpha \varphi$  if $\varphi(0) = \varphi^{(1)}(0) = \cdots = \varphi^{(n -1)}(0) = 0$.  
It is easy to show that, for any $\alpha,\beta \geq0$, 
$$
I^{\alpha+\beta}\varphi(t)=I^{\alpha}I^{\beta}\varphi(t), \quad
D^\alpha D^\beta\varphi = D^{\alpha+\beta}\varphi,
$$
and
$$ 
D^\alpha I^\beta \varphi = \begin{cases} D^{\alpha-\beta}\varphi &\mbox{ if }\alpha>\beta \\ I^{\beta-\alpha}\varphi &\mbox{ if }\alpha\leq\beta\end{cases}
$$
Furthermore, if $\varphi$ is sufficiently smooth (say, $\varphi\in C^{n}([0,T]))$  and $\varphi(0)=\cdots =\varphi^{(n-1)}(0)=0$, then
\begin{equation}
    \label{eqn 10.1}
I^{\alpha}_t\partial^{\alpha}_t \varphi(t): = I^{\alpha}_t (\partial^{\alpha}_t \varphi) (t)=\varphi(t), \quad \forall \, t\in [0,T].
\end{equation}
Consequently, if $\varphi\in C^{2}([0,T])$ and $\alpha\in (0,2)$, then $\partial^{\alpha}_t \varphi=f$  is equivalent to
\begin{equation}
   \label{eqn 4.29.5}
   \varphi(t)-\varphi(0)-1_{\alpha>1}\varphi'(0)t=I^{\alpha}_t f(t), \quad \forall\,  t\in [0,T].
   \end{equation}

Now we introduce the class of weights used in this article.
\begin{definition}[$A_p$-weight]\label{def 05.06.1}
Let $1<p<\infty$. We write $w\in A_p$ if $w(x)$ is a nonnegative measurable function on $\bR^d$ such that
\begin{equation*}
[w]_{p}:=\sup_{x_{0}\in\bR^d,r>0}\left(\aint_{B_{r}(x_{0})}w(x)dx\right)\left(\aint_{B_{r}(x_{0})}w(x)^{-1/(p-1)}dx\right)^{p-1}<\infty,
\end{equation*}
where 
\begin{equation*}
\aint_{B_{r}(x_{0})}w(x)dx=\frac{1}{|B_{r}(x_{0})|}\int_{B_{r}(x_{0})}w(x)dx.
\end{equation*}
If $w\in A_{p}$, then $w$ is said to be an $A_p$ weight.
\end{definition}

\begin{remark} \label{rmk 06.20.1}
It is well known that the Hardy-Littlewood maximal function 
 is bounded in $L_p(w dx)$ if any only if $w$ is an $A_p$ weight (see e.g.  \cite{grafakos2009modern}). Therefore, if one uses an approach based on   sharp and maximal functions, then it is natural to consider $L_p$-spaces with $A_p$ weights for full generality.
\end{remark}

\begin{remark}\label{rmk 09.27.21:52}
The class $A_{p}$  is increasing as $p$ increases, and it holds that
\begin{equation*}
A_{p}=\bigcup_{q\in(1,p)} A_{q}.
\end{equation*}
More precisely, for any $w\in A_{p}$, one can find $q<p$, which depends on $d,p$, and $[w]_{p}$ such that $w\in A_{q}$.
\end{remark}


Let $p,q\in(1,\infty)$, $n\in\bN$, and $T\in(0,\infty]$. For $w_{1} = w_{1}(x)\in A_{p}$ and $w_{2}=w_{2}(t)\in A_{q}$, we define
$$
L_{p}(w_{1})=L_{p}(\R^{d},w_{1}dx),\quad H_p^2(w_{1}) = \{ u\in L_p(w_{1}):\, D^{\gamma}u\in L_{p}(w_{1}),\,|\gamma|\leq 2 \},
$$
$$
\bL_{q,p}(w_{2},w_{1},T)=L_{q}((0,T), w_{2}dt ; L_{p}(w_{1})),
$$
and
$$
\bH^{0,2}_{q,p}(w_{2},w_{1},T)=L_{q}((0,T),w_{2}dt ; H^{2}_{p}(w_{1})).
$$
We omit $T$ if $T = \infty$. For example,
\begin{equation*}
\bL_{q,p}(w_{2},w_{1}) = \bL_{q,p}(w_{2},w_{1},\infty) = L_q((0,\infty),w_{2}dt ;L_p(w_{1})).
\end{equation*}
The norms of these function spaces are defined in a natural way. For example,
\begin{equation*}
\begin{aligned}
\|f\|_{\bL_{q,p}(w_{2},w_{1},T)}&:=\left(\int_0^T \|f(t,\cdot)  \|_{L_p(w_{1})}^qw_{2}(t) dt\right)^{1/q} \\
&=\left(\int_{0}^{T}\left( \int_{\R^{d}}|f(t,x)|^{p}w_{1}(x)dx \right)^{q/p}w_{2}(t)dt\right)^{1/q}.
\end{aligned}
\end{equation*}

\begin{remark}\label{rmk 05.08.1}

 (i). Suppose that $w\in A_{p}$. Then the weighted Sobolev spaces have similar properties as the usual $L_p$-spaces. For example, one can check that $C_c^\infty(\R^{d})$ is dense in $H_p^2(w)$ and $H_p^2(w)$ is a Banach space. Also,  the dual space of $L_p(w)$ is $L_{p'}(\tilde{w})$, where  
\begin{equation} \label{dual}
1/p+1/{p'}=1,\quad \tilde{w}=w^{-\frac{1}{p-1}}.
\end{equation}
In other words,  for any bounded linear functional $\Lambda$ defined on $L_p(w)$ there is a unique $g\in L_{p'}(\tilde{w})$ such that $\Lambda f=\int_{\bR^d}fg \,dx$ for any $f\in L_p(w)$, and $\|\Lambda\|=\|g\|_{L_{p'}(\tilde{w})}$. 

(ii). Let $\phi \in C^{\infty}_c(B_1(0))$ and $\phi_{\varepsilon}(x):=\varepsilon^{-d}\phi(x/\varepsilon)$, $\varepsilon \in (0,1]$.  Then, it is easy to check that 
$$
|\phi_{\varepsilon}(x)|\leq \frac{\|\phi\|_{L^{\infty}}}{|x|^d},  \quad \forall \, x\in \bR^d,
$$
$$
|\phi_{\varepsilon}(x)-\phi_{\varepsilon}(y)|\leq \frac{4^{d+1} \|D\phi\|_{L^{\infty}} |x-y|}{(|x|+|y|)^{d+1}} \quad \text{if} \quad 2|x-y|\leq \max\{|x|, |y|\}.
$$
By \cite[Corollary 9.4.7]{grafakos2009modern},   for any $f\in L_{p}(w)$, it holds that
\begin{equation*}
\|\int_{\R^{d}}f(y)\phi_{\varepsilon}(\cdot-y)dy\|_{L_{p}(w)}\leq N_0 \|f\|_{L_{p}(w)}, \quad \forall \varepsilon \in (0,1],
\end{equation*}
where the constant $N_0$ depends only on $d,p,[w]_{p}$, and $\|\phi\|_{L_{1}}+\|\phi\|_{L_{\infty}}+\|D\phi\|_{L_{\infty}}$. This implies that for $f\in L_{p}(w)$, the convolution
\begin{equation*}
f^{\varepsilon}(x):=\varepsilon^{-d}\int_{\R^{d}}f(y)\phi((x-y)/\varepsilon)dy
\end{equation*}
converges to $f$ in $L_{p}(w)$ as $\varepsilon \downarrow 0$. Indeed, for given small $\delta>0$, take $h\in C^{\infty}_c(\bR^d)$ such that $\|f-h\|_{L_p(w)}<\delta$. Then
\begin{eqnarray*}
\|f-f^{\varepsilon}\|_{L_p(w)}&\leq &\|f-h\|_{L_p(w)}+ \|f^{\varepsilon}-h^{\varepsilon}\|_{L_p(w)}+\|h-h^{\varepsilon}\|_{L_p(w)}\\
&\leq& \delta+N_0\delta +\|h-h^{\varepsilon}\|_{L_p(w)}.
\end{eqnarray*}
The last term above converges to $0$ as $\varepsilon \downarrow 0$ since $h^{\varepsilon}\to h$ uniformly on $\bR^d$.

\end{remark}

Now we introduce our solution space and related facts.

\begin{definition}[Solution space]\label{def 05.07.2}
Let $0<\alpha<2$, $1<p,q<\infty$, $w_{1}\in A_{p}(\R^{d})$, $w_{2}\in A_{q}(\R)$, and $T<\infty$. We write $u\in \bH^{\alpha,2}_{q,p}(w_{2},w_{1},T)$ if there exists a defining sequence $u_{n}\in C^{\infty} ([0,\infty)\times\R^{d})$ such that $u_{n}$ converges to $u$ in $\bH^{0,2}_{q,p}(w_{2},w_{1},T)$, and $\partial^{\alpha}_{t}u_{n}$ is a Cauchy in $\bL_{q,p}(w_{2},w_{1},T)$. For $u \in \bH_{q,p}^{\alpha,2}(w_{2},w_{1},T)$,  we write
$$
f=\partial^{\alpha}_t u
$$
if $f$ is the limit of $\partial^{\alpha}_{t}u_{n}$ in $\bL_{q,p}(w_{2},w_{1},T)$. For $u\in \bH^{\alpha,2}_{q,p}(w_{2},w_{1},T)$, we write  $u\in \bH^{\alpha,2}_{q,p,0}(w_{2},w_{1},T)$ if there is a defining sequence $u_{n}\in C^{\infty}([0,\infty)\times\R^{d})$ such that 
$$
u_n(0,x)=1_{\alpha>1}\partial_{t}u(0,x)=0.
$$ 
Finally, define
$$
\bH^{\alpha,2}_{p,p,0}(T)=\bH^{\alpha,2}_{p,p,0}(1,1,T).
$$
\end{definition}

\begin{lemma}\label{lem 05.08.1}
 Let $1<p,q<\infty$, $w_{1}\in A_{p}(\R^{d})$, $w_{2}\in A_{q}(\R)$, and $T<\infty$.
 
 (i) The spaces $\bH^{\alpha,2}_{q,p}(w_{2},w_{1},T)$ and  $\bH^{\alpha,2}_{q,p,0}(w_{2},w_{1},T)$ are Banach spaces with respect to the norm
\begin{equation*}
\|u\|_{\bH^{\alpha,2}_{q,p}(w_{2},w_{1},T)}:=\|u\|_{\bH^{0,2}_{q,p}(w_{2},w_{1},T)}+\|\partial^{\alpha}_{t}u\|_{\bL_{q,p}(w_{2},w_{1},T)}.
\end{equation*}

(ii) $C^{\infty}_c((0,\infty)\times \bR^d)$ is dense in $\bH^{\alpha,2}_{q,p,0}(w_{2},w_{1},T)$. 
\end{lemma}

\begin{proof}
(i) It can be readily proved by following a straightforward argument. 
\\
(ii) Let $u\in \bH^{\alpha,2}_{q,p,0}(w_{2},w_{1},T)$ be given. We will construct a defining sequences belonging to $C^{\infty}_c((0,\infty)\times \bR^d)$. First, note that 
by the definition of $\mathbb{H}^{\alpha,2}_{q,p,0}(w_{2},w_{1},T)$, we may assume $u\in C^{\infty}([0,\infty)\times\R^{d})$ and satisfies
$u(0,x)=1_{\alpha>1}\partial_{t}u(0,x)=0$.  Also considering multiplications with smooth cut-off functions depending only on $x$, we may further assume that $u$ has compact support with respect to $x$. 

Extend $u(t,x)=0$ and $f(t,x)=f(0,x)$ for $t<0$. Then, for any multi-index $\alpha$,  $D^{\alpha}_{x} u(t,x)$ and $f(t,x)$ are  continuous in $t\in \bR$ since $D^{\alpha}_{x} u(0,x)=0$.  Take $\eta_{1}\in \Ccinf((1,2))$ with the unit integral and let
\begin{equation*}
\begin{aligned}
u^{\varepsilon}(t,x)=\varepsilon^{-1}\int_{\R}u(s,x)\eta_{1}((t-s)/\varepsilon)ds,
\\
f^{\varepsilon}(t,x)=\varepsilon^{-1}\int_{\R}f(s,x)\eta_{1}((t-s)/\varepsilon)ds.
\end{aligned}
\end{equation*}
Then $u^{\varepsilon}(t)=0$ if $t<\varepsilon$,  $u^{\varepsilon}$ is infinitely differentiable in $(t,x)$ and 
\begin{equation}
\label{eqn 10.19}
|u^{\varepsilon}(t,x)-u(t,x)|+ |Du^{\varepsilon}(t,x)-Du(t,x)|+|D^2u^{\varepsilon}(t,x)-D^2u(t,x)| \to 0
\end{equation}
as $\varepsilon \downarrow 0$,  uniformly on $[0,T]\times \bR^d$.  Similarly, $f^{\varepsilon}(t,x) \to f(t,x)$ uniformly on $[0,T]\times \bR^d$.  Let $\eta_{2}\in \Ccinf(\R)$ so that $\eta_{2}(t)=1$ for $t\leq T$, and $\eta_{2}=0$ for $t\geq T+1$. Set 
$$v^{\varepsilon}:=\eta_{2}u^{\varepsilon}\in\Ccinf((0,\infty)\times\R^{d}).
$$
 Let $n$ be an integer so that $n-1\leq \alpha<n$, then we have
\begin{equation*}
\begin{aligned}
I^{n-\alpha}_{t}u^{\varepsilon}(t)&=\frac{1}{\Gamma(n-\alpha)}\int_0^t(t-s)^{n-\alpha-1}u^{\varepsilon}(s)ds
\\
&=\frac{1}{\Gamma(n-\alpha)}\int_0^t(t-s)^{n-\alpha-1}\varepsilon^{-1}\int_{\R}u(s-r)\eta_{1}(r/\varepsilon)drds
\\
&=\frac{1}{\Gamma(n-\alpha)}\int_{\R}\varepsilon^{-1}\left(\int_{0}^{t-r}(t-r-s)^{n-\alpha-1}u(s)ds \right)\eta_{1}(r/\varepsilon)dr.
\end{aligned}
\end{equation*}
Taking derivative $(d/dt)^{n}$ to $I^{n-\alpha}_t u^{\varepsilon}(t)$  and using \eqref{Caputo_Riemann} (recall that $u(0)=1_{1<\alpha}\partial_{t}u(0)=0$), we have 
$$\partial^{\alpha}_{t}u^{\varepsilon}=f^{\varepsilon}.
$$
 Also, since $\eta_{2}=1$ if $t\leq T$, we have $\partial^{\alpha}_{t}v^{\varepsilon}=f^{\varepsilon}$ for all $t\leq T$.
Moreover, as $\varepsilon \downarrow 0$, by  \eqref{eqn 10.19} and the dominated convergence theorem,
$$v^{\varepsilon}\to u \quad \text{in} \quad\mathbb{H}^{0,2}_{q,p}(w_{2},w_{1},T),
$$
$$ f^{\varepsilon}\to f \quad \text{in} \quad  \mathbb{L}_{q,p}(w_{2},w_{1},T).
$$
Therefore, $v^{\varepsilon}\in\Ccinf((0,\infty)\times\R^{d}) $ converges to $u$ in $\mathbb{H}^{\alpha,2}_{q,p,0}(w_{2},w_{1},T)$. The lemma is proved.
\end{proof}

Recall (see \eqref{eqn 4.29.5}) that if $u(\cdot,x)\in C^2([0,T])$ then
$$
 u(t,x)-u(0,x)-1_{\alpha>1}u'(0,x)t=I^{\alpha}_t \partial^{\alpha}_t u(t,x), \quad \forall\,  t\in [0,T].
 $$
 If $u\in \bH^{\alpha,2}_{q,p,0}(w',w,T)$ then 
 $$
 u(0,x)=0, \quad 1_{\alpha>1}\partial_{t}u(0,x)=0
 $$
 in the following sense. (The result of Proposition \ref{lem 05.08.2} is not used elsewhere)

\begin{proposition}\label{lem 05.08.2}
Let $u\in\mathbb{H}^{\alpha,2}_{q,p}(w_{2},w_{1},T)$ and $f\in \mathbb{L}_{q,p}(w_{2},w_{1},T)$. Then the following are equivalnet:
\begin{enumerate}[(i)]
\item
$u\in\mathbb{H}^{\alpha,2}_{q,p,0}(w_{2},w_{1},T)$, and $\partial^{\alpha}_{t}u=f$.
\item
For any $\phi\in\Ccinf(\R^{d})$
\begin{equation}\label{equivalent_def_frac_deriv_2}
(u(t),\phi)=I^{\alpha}_{t}(f(t),\phi),\quad t\leq T \quad (a.e.).
\end{equation}
\end{enumerate}
\end{proposition}
\begin{proof}
\textbf{Step 1.} Suppose (i) holds. Then, by Lemma \ref{lem 05.08.1}(ii), there exists  $u_{n}\in\Ccinf((0,\infty)\times\R^{d})$ such that $u_n\to u$ in $\bH^{\alpha,2}_{q,p.0}(w_{2},w_{1},T)$. Note that since $w_{2}\in L_{1}([0,T])$ and $w_{2}>0$  almost everywhere, it follows that 
\begin{equation*}
\int_{E}w_{2}(t)dt=0
\end{equation*}
if and only if $|E|=0$. Therefore, by Remark \ref{rmk 05.08.1}, for any $\phi\in\Ccinf(\R^{d})$
\begin{eqnarray*}
&&|(u_n(t)-u(t),\phi)|+|(\partial^{\alpha}_t u_n(t)-\partial^{\alpha}_t u, \phi)|\\
&\leq& \|\phi\|_{L_{p'}(\tilde{w})} \left(\|u_n(t)-u(t)\|_{L_p(w)}
+\|\partial^{\alpha}_t u_n(t)-\partial^{\alpha}_t u(t)\|_{L_p(w)}\right) \to 0
\end{eqnarray*}
as $n \to \infty$ ($w_{2}(t)dt$-a.e., or equivalently $dt$-a.e.). Therefore, for almost all $t\leq T$, 
\begin{equation*}
\begin{aligned}
(u(t),\phi) &= \lim_{n\to\infty}(u_n(t),\phi)
=\lim_{n\to\infty} \int_{\R^{d}} I_t^\alpha \partial_t^\alpha u_n(t,x) \phi(x) dx 
= I_t^\alpha (  \partial_t^\alpha u(t,\cdot),\phi).
\end{aligned}
\end{equation*}
In the second equality above we used \eqref{eqn 10.1}. Therefore, we have \eqref{equivalent_def_frac_deriv_2}.

\textbf{Step 2.} There is a version of proof in \cite[Remark 2.9]{kim16timefractionalspde} for the case when $w_{2}=w=1$. The general case can be proved by repeating the proof of  \cite[Remark 2.9]{kim16timefractionalspde} and using   Remark \ref{rmk 05.08.1}. 
\end{proof}

Finally, we introduce our main result.

\begin{theorem} \label{theorem 5.1}
Let $0<\alpha<2$, $1<p,q<\infty$, $w_{1}=w_{1}(x)\in A_{p}(\R^{d})$, $w_{2}=w_{2}(t)\in A_{q}(\R)$, $T<\infty$, and $f\in \bL_{q,p}(w_{2},w_{1},T)$. Then the equation
\begin{equation} \label{eqn 05.09.1}
\partial^{\alpha}_{t}u=\Delta u +f, \quad t>0\,; \quad u(0,\cdot)=1_{\alpha>1}\partial_{t}u(0,\cdot)=0
\end{equation} 
has a unique solution $u$ in $\bH_{q,p,0}^{\alpha,2}(w_{2},w_{1},T)$. Also, the solution $u$ satisfies
\begin{equation} \label{eqn 05.08.1-1}
\|u_{xx}\|_{\bL_{q,p}(w_{2},w_{1},T)}\leq N_0\|f\|_{\bL_{q,p}(w_{2},w_{1},T)},
\end{equation}
\begin{equation} \label{eqn 05.08.1}
\|u\|_{\bH_{q,p}^{\alpha,2}(w_{2},w_{1},T)}\leq N_1\|f\|_{\bL_{q,p}(w_{2},w_{1},T)},
\end{equation}
where  $N_0=N_0(\alpha,d,p,q,[w_{1}]_{p}, [w_{2}]_{q})$, and $N_1=N_1(\alpha,d,p,q,[w_{1}]_{p}, [w_{2}]_{q},T)$.
\end{theorem}

\mysection{Sharp function estimates of solutions}

In this section we prove a priori estimates  \eqref{eqn 05.08.1-1} and \eqref{eqn 05.08.1}  based on the sharp function estimates of solutions.

Let $p(t,x)$ be the fundamental solution to the following equation on $\bR^d$.
\begin{equation}\label{3.1}
\partial^\alpha_t u(t,x) = \Delta u(t,x),\quad t>0\,; \quad u(0) = u_0,\,\, 1_{\alpha>1}\partial_{t}u(0) = 0.
\end{equation}
In other words, if $u_0$ is smooth enough then the solution to equation \eqref{3.1} is given by 
\begin{equation}
 \label{fundamental}
u(t,x) = p(t,\cdot)\ast u_0  = \int_{\bR^d}p(t,y)u_0(x-y)dy.
\end{equation}
It is well known (see e.g. Lemma  \ref {prop_of_p_q} below) that such $p$ exists and absolutely continuous in $t$. Define
\begin{equation*}
q(t,x) = \begin{cases}  I_t^{\alpha-1}p(t,x)~&\alpha\in(1,2), \\ p(t,x)~&\alpha=1, \\ D_t^{1-\alpha}p(t,x)~&\alpha \in(0,1). \end{cases}
\end{equation*}

In the following two lemmas we collect  some properties of $p(t,x)$ and $q(t,x)$. 
 
\begin{lemma} \label{prop_of_p_q}
\begin{enumerate}[(i)]
\item There exists a fundamental solution $p(t,x)$.   Moreover, for all $t\neq0$ and $x\neq0$, we have
\begin{equation*}
\partial_t^\alpha p(t,x) = \Delta p(t,x),\quad \frac{\partial p(t,x)}{\partial t} = \Delta q(t,x).
\end{equation*}
Also, for each $x\neq0$,  $\frac{\partial}{\partial t}p(t,x)\to0$ as $t\downarrow0$. Furthermore, $\frac{\partial}{\partial t}p(t,\cdot)$ is integrable in $\mathbb{R}^d$ uniformly on $t\in[\varepsilon,T]$ for any $\varepsilon>0$.
\item Let $d\geq1$, $\alpha\in(0,2)$, $m,n=0,1,2,\cdots$, and $R = t^{-\alpha}|x|^2$. Then there exist constants $N$ and $\sigma>0$ depending only on $m,n,d$ and $\alpha$ so that if $R\geq1$
\begin{equation}\label{bounds_of_q_2}
|\partial_t^nD_x^m q(t,x)| \leq N t^{\frac{-\alpha(d+m)}{2}-n+\alpha-1}\exp{\{-\sigma t^{-\frac{\alpha}{2-\alpha}}|x|^{\frac{2}{2-\alpha}}\}},
\end{equation}
and if $R\leq1$
\begin{equation} \label{bounds_of_q_1}
\begin{aligned}
|\partial_t^nD_x^m q(t,x)| &\leq N|x|^{-d-m}t^{-n+\alpha-1}(R^2+R^2|\ln R|\cdot 1_{d=2}) \\
& \quad +N|x|^{-d}t^{-n+\alpha-1}(R^{1/2} 1_{d=1}+R1_{d=2}+R^2|\ln R|\cdot 1_{d=4})1_{m=0}.
\end{aligned}
\end{equation}
In particular, for each $n\in\mathbb{N}$,
\begin{equation} \label{bounds_of_q_3}
|D_x^n q(1,x)| \leq N(d,\alpha,n)(|x|^{-d+2-n}\wedge|x|^{-d-n}).
\end{equation}
\item There is a scaling property of $q(t,x)$, i.e., 
\begin{equation}\label{scaling_of_q}
q(t,x) = t^{-\frac{\alpha d}{2}+\alpha-1}q(1,xt^{-\frac{\alpha}{2}}).
\end{equation}
\end{enumerate}
\end{lemma}
\begin{proof}
See e.g.  \cite[Lemma 3.2]{kim17timefractionalpde} for (i)-(ii), and  see  equality (5.2) of \cite{kim2015asymptotic} for (iii).
\end{proof}

\begin{lemma}\label{representation_whole_space}
\begin{enumerate}[(i)]
\item \label{representation} Let $u\in\Ccinf((0,\infty)\times\R^{d})$, and let $f:=\partial^{\alpha}_{t}u-\Delta u$. Then
\begin{equation*}
u(t,x)=\int_{0}^{t}\int_{\R^{d}}q(t-s,x-y)f(s,y)dyds.
\end{equation*}
\item \label{representation2} Let $p>1$, $T<\infty$ and $f\in C_c^\infty((0,\infty)\times\R^{d})$. Define 
\begin{equation} \label{sol_def}
u(t,x) = \int_0^t\int_{\mathbb{R}^d} q(t-s,x-y)f(s,y) dyds,
\end{equation}
where $(t,x)\in(0,T)\times\bR^d$. Then $u$ is the unique solution to the equation  
\begin{equation*}
\partial^\alpha_t u = \Delta u+f, \quad 0<t<T,\, x\in \bR^d\,; \quad u(0)=1_{\alpha>1} \partial_t u(0)=0
\end{equation*}
in the class $\bH^{\alpha,2}_{p,p,0}(T)$. Moreover,
\begin{equation}\label{L_2 estimate}
\|D^{2}u\|_{\bL_{p}(T)}\leq N \|f\|_{\bL_{p}(T)},
\end{equation}
where $N$ depends only on $\alpha,d$, and $p$. In particular, we have
\begin{equation*}
\int_0^\infty \| D^2u(t,\cdot) \|_{L_p}^p dt \leq N \int_0^\infty \| f(t,\cdot) \|_{L_p}^p dt.
\end{equation*} 
\end{enumerate}
\end{lemma}
\begin{proof}
\eqref{representation} and \eqref{representation2} are consequences of  Lemma 3.5 and Theorem 2.10 of \cite{kim17timefractionalpde}. We only remark that  the independency of $N_0$ on $T$ can be easily checked based on the estimate \eqref{L_2 estimate} for $T=1$ and considering  the equation for $\bar{u}(t,x):=u(Tt,T^{\alpha/2}x)$. For the equation of $\bar{u}$, use   the relation 
\begin{equation}
\label{eqn 4.1.7}
\partial^{\alpha}_t (h(ct))=c^{\alpha}(\partial^{\alpha}_t h)(ct).
\end{equation}
\end{proof}


\begin{lemma}\label{useful_lemma}
\begin{enumerate}[(i)]
\item \label{int_by_part}  Let $r<0$ and $f\in C_c^\infty(\bR^d\setminus \{0\})$. Then we have
\begin{equation}\label{eqn 5.03.1}
\int_{\R^{d}}|z|^{r}f(z)dz=-r\int_{0}^{\infty}\rho^{r-1}\int_{|z|<\rho}f(z)dzd\rho.
\end{equation}
\item \label{bound_of_q_4} Let $\gamma=(\gamma_{1},\dots,\gamma_{d})$ be a multi-index satisfying $|\gamma|\leq 2$. Then for $\ep\in[0,1]$,
\begin{equation} \label{useful_bound}
|D^\gamma q(t,x)|\leq N(\alpha,d,\ep)|t|^{(1-|\gamma|/2+\ep/2)\alpha -1}|x|^{-d-\ep}.
\end{equation}
Furthermore, if $|\gamma|\leq 1$ then \eqref{useful_bound} is also vaild for $\ep\in[-1,0)$.
\end{enumerate}
\end{lemma}
\begin{proof}
\eqref{int_by_part} Let $F(z)$ and $G(z) = G(|z|)$ be smooth functions on $\bR^d\setminus \{0\}$. For $0<\varepsilon<R<\infty$, we have
\begin{equation*}
\begin{aligned}
\int_{\varepsilon<|z|<R}F(z)G(|z|)dz&=\int_{\varepsilon}^{R}\int_{\partial B_{\rho}(0)}G(\rho)F(\sigma)d\sigma_{\rho} d\rho
\\
&=\int_{\varepsilon}^{R}G(\rho)\left(\int_{\partial B_{\rho}(0)}F(\sigma)d\sigma_{\rho}\right) d\rho
\\
&=\int_{\varepsilon}^{R}G(\rho)\frac{d}{d\rho}\left(\int_{B_{\rho}(0)}F(z)dz\right) d\rho.
\end{aligned}
\end{equation*}
By applying the integration by parts to the last term, we have
\begin{equation} \label{int_by_part1}
\begin{aligned}
\int_{\varepsilon<|z|<R}F(z)G(|z|)dz&=G(R)\int_{|z|<R}F(z)dz-G(\varepsilon)\int_{|z|<\varepsilon}F(z)dz\\
&\quad -\int_{\varepsilon}^{R}G'(\rho)\int_{|z|<\rho}F(z)dzd\rho.
\end{aligned}
\end{equation}
Applying \eqref{int_by_part1} to $G(z)=|z|^{r}$ and $F(z)=f(z)$,  we have
\begin{equation*}
\begin{aligned}
\int_{\varepsilon<|z|<R}|z|^{r}f(z)dz&=R^{r}\int_{|z|<R}f(z)dz-\varepsilon^{r}\int_{|z|<\varepsilon}f(z)dz\\
&\quad -r\int_{\varepsilon}^{R}\rho^{r-1}\int_{|z|<\rho}f(z)dzd\rho.
\end{aligned}
\end{equation*}
By letting $\varepsilon \downarrow 0$, and $R\to \infty$, we have \eqref{eqn 5.03.1}. The condition that $f$ vanishes near $x=0$ is used when $\varepsilon \downarrow 0$, and the condition $r<0$ is used to have $R^{r}\to 0$ as $R\to \infty$.

\eqref{bound_of_q_4} Let $b,c>0$ and $\ep\in[-1,1]$. Observe that
\begin{equation} \label{exp_bound}
\begin{gathered}
e^{-br^{c}}\leq N(b,c,d, \ep)r^{-d-\ep},\quad\forall r>0,
\\
r^{b}|\log{r}|\leq N(b), \quad \forall r\leq 1.
\end{gathered}
\end{equation} 
Now let $|\gamma|\leq 2$ and $\ep\in[0,1]$. Then, by \eqref{scaling_of_q}, \eqref{bounds_of_q_2}, \eqref{bounds_of_q_1}, and \eqref{exp_bound}, we have
\begin{equation*}
\begin{aligned}
&|D^\gamma q(t,x)| \\
&= |t|^{(-\frac{d}{2}+1-\frac{|\gamma|}{2})\alpha-1}|(D^\gamma q)(1,t^{-\alpha/2}x)| \\
&\leq N|t|^{(-\frac{d}{2}+1-\frac{|\gamma|}{2})\alpha-1}(1_{t^{-\alpha/2}|x|\leq1}(t^{-\alpha/2}|x|)^{-d} + 1_{t^{-\alpha/2}|x|\geq1}(t^{-\alpha/2}|x|)^{-d-\ep}) \\
&\leq N|t|^{(-\frac{d}{2}+1-\frac{|\gamma|}{2})\alpha-1}(1_{t^{-\alpha/2}|x|\leq1}(t^{-\alpha/2}|x|)^{-d-\ep} + 1_{t^{-\alpha/2}|x|\geq1}t^{\alpha(d+\ep)/2 }|x|^{-d-\ep}) \\
&\leq N|t|^{(1-\frac{|\gamma|}{2}+\frac{\ep}{2})\alpha-1}|x|^{-d-\ep}. \\
\end{aligned}
\end{equation*}
To show the second assertion, let $|\gamma|\leq1$ and $\ep\in[-1,0)$. Again, by \eqref{scaling_of_q}, \eqref{bounds_of_q_2}, \eqref{bounds_of_q_1}, and \eqref{exp_bound}, we have
\begin{equation*}
\begin{aligned}
&|D^\gamma q(t,x)| \\
&= |t|^{(-\frac{d}{2}+1-\frac{|\gamma|}{2})\alpha-1}|(D^\gamma q)(1,t^{-\alpha/2}x)| \\
&\leq N|t|^{(-\frac{d}{2}+1-\frac{|\gamma|}{2})\alpha-1}(1_{t^{-\alpha/2}|x|\leq1}(t^{-\alpha/2}|x|)^{-d-\ep} + 1_{t^{-\alpha/2}|x|\geq1}(t^{-\alpha/2}|x|)^{-d-\ep}) \\
&\leq N|t|^{(1-\frac{|\gamma|}{2}+\frac{\ep}{2})\alpha-1}|x|^{-d-\ep}. \\
\end{aligned}
\end{equation*}
The lemma is proved.
\end{proof}

For a real-valued measurable function $h$ on $\mathbb{R}^{d+1}$, define the maximal function 
$$ \mathbb{M}h(t,x) :=\sup_{E}\frac{1}{|E|}\int_{E}|h(s,y)|dsdy,
$$
where the supremum is taken over the cubes $E$  of the form
\begin{equation*}
E=[r,s]\times[a^{1},b^{1}]\times\dots\times[a^{d},b^{d}]
\end{equation*}
 containing $(t,x)$.

For $f\in C_{c}^\infty(\R^{d+1})$, $T\in (0,\infty]$,  $i,j=1,2,\dots,d$ and $(t,x)\in \bR^{d+1}$,  we define
$$
L_{0}f(t,x)=L^T_{0}f(t,x)=\int_{-\infty}^t\int_{\R^{d}}1_{0<t-s<T}q(t-s,x-y)f(s,y)dyds, 
$$
$$
L^{i}_{1}f(t,x)=L^{T,i}_{1}f(t,x)=\int_{-\infty}^t\int_{\R^{d}}1_{0<t-s<T}D_{i}q(t-s,x-y)f(s,y)dyds,
$$
$$
L^{ij}_{2}f(t,x)=\int_{-\infty}^t\int_{\R^{d}}D_{ij}q(t-s,x-y)f(s,y)dyds,
$$
and denote
$$
L_{1}f=L^T_1f=(L^{T,1}_1f,\cdots, L^{T,d}_1 f), \quad L_2f=(L^{ij}_2f)_{i,j=1,2,\cdots, d}.
$$
Observe that
 for any $(t_{0},x_{0})\in\R^{d+1}$,
\begin{equation}\label{eqn 05.16.1}
\begin{aligned}
L_{k}f(t+t_{0},x+x_{0})&=L_{k}(f(t_{0}+\cdot,x_{0}+\cdot))(t,x),\quad (k=0,1,2),
\end{aligned}
\end{equation}  
and 
for any $c>0$,
\begin{equation}\label{scaling_qxx}
\begin{aligned}
L_{2}(f(c^{\frac{2}{\alpha}}\cdot,c\cdot))(t,x)&=L_{2}f(c^{\frac{2}{\alpha}}t,cx).
\end{aligned}
\end{equation}

For $w_{1}\in A_{p}(\bR^d)$ and $w_{2}\in A_{q}(\bR)$, and $T\in (0,\infty]$, define 
\begin{equation*}
\tilde{\bL}(q,p,w_{2},w_{1},T) := L_q\big((-\infty,T), w_{2}dt; L_p(w_{1})),
\end{equation*}
where
$$
\|f\|_{\tilde{\bL}(q,p,w_{2},w_{1},T)}:=\left(\int_{-\infty}^{T}\left(\int_{\R^{d}}|f(t,x)|^{p}w_{1}(x)dx\right)^{q/p}w_{2}(t)dt\right)^{1/q}.
$$
As usual, we omit $T$ if $T = \infty.$ 

Here is  the main result of this section.
\begin{theorem}\label{thm 05.06.3}
Let $p,q\in(1,\infty)$, $T\in(0,\infty)$, $w_{1}\in A_{p}(\R^{d})$ and $w_{2}\in A_{q}(\R)$. Then for any  $f\in C^{\infty}_c(\bR^{d+1})$ we have
\begin{equation} \label{Lp_estimate}
\|L_{k}f\|_{\tilde{\bL}(q,p,w_{2},w_{1},T)}\leq N_k \|f\|_{\tilde{\bL}(q,p,w_{2},w_{1},T)}, \quad k=0,1,2
\end{equation}
where $N_{k}=N_k(\alpha,d,p,q,[w_{1}]_{p},[w_{2}]_{q},T)$ $(k=0,1)$, and $N_2=N_2$$($$\alpha,d,p,q$,
$[w_{1}]_{p}$,$[w_{2}]_{q}$$)$. In particular, $N_2$ is independent of $T$.
\end{theorem} 

\begin{remark}
 \label{remark 10.24.1}
Since $C^{\infty}_c(\bR^{d+1})$ is dense in $\tilde{\bL}(q,p,w_{2},w_{1},T)$, by Theorem \ref{thm 05.06.3},  the operators $L_k$ can be continuously extended to $\tilde{\bL}(q,p,w_{2},w_{1},T)$.
\end{remark}

For $\delta>0$, define 
\begin{equation*}
Q_{\delta}:=[-\delta^{\frac{2}{\alpha}},0]\times[-\delta/2,\delta/2]^{d}.
\end{equation*}

\begin{lemma} \label{2_2_bound_1}
Let  $p_0\in(1,\infty)$, $T<\infty$ and $f\in C_c^\infty(\bR^{d+1})$. Assume that $f=0$ outside of $[-(2\delta)^{2/\alpha}, (2\delta)^{2/\alpha}]\times B_{3\delta d/2}$. Then for  $(t,x)\in Q_{\delta}$,
\begin{equation} \label{bound_1_ineq}
\begin{gathered}
\aint_{Q_{\delta}}| L_{k}f(s,y) |^{p_{0}}dyds \leq N_{k}(\alpha,d,p_0,T) \mathbb{M}|f|^{p_{0}}(t,x),\quad k = 0,1, \\
\aint_{Q_{\delta}}| L_{2}f(s,y) |^{p_{0}}dyds \leq N_{2}(\alpha,d,p_0) \mathbb{M}|f|^{p_{0}}(t,x).
\end{gathered}
\end{equation}
\end{lemma}

\begin{proof} 
\textbf{Step 1.} $(k = 0,1)$.  Let $k=0$. By Minkowski's inequality,
\begin{equation*}
\begin{aligned}
&\aint_{Q_{\delta}}  |L_{0}f(s,y)|^{p_{0}}  dyds
\\
&\quad \leq N(d)\delta^{-2/\alpha-d}\left( \int_0^T\int_{\bR^d} \left( \int_{Q_\delta}  |q(r,z)f(s-r,y-z)|^{p_0}  dyds\right)^{1/p_0} dzdr \right)^{p_0} \\
&\quad \leq N(d)\delta^{-2/\alpha-d}I_0^{p_0} \int_{\bR^{d+1}} |f(s,y)|^{p_0} dyds, \\
&\quad \leq N(d) I_0^{p_0} \aint_{[-(2\delta)^{2/\alpha},(2\delta)^{2/\alpha}]\times B_{3\delta d/2} } |f(s,y)|^{p_0} dyds, \\
\end{aligned}
\end{equation*}
where
\begin{equation*}
I_0 = I_0(T) = \int_0^T\int_{\bR^d} |q(r,z)|dzdr.
\end{equation*}
By \eqref{scaling_of_q}, change of variables, \eqref{bounds_of_q_2} and \eqref{bounds_of_q_1}, we have
\begin{equation} \label{during_proof_11}
\begin{gathered}
\int_0^T\int_{\bR^d} |q(r,z)|dzdr = \int_0^T |r|^{\alpha-1} dr\int_{\bR^d} |q(1,z)|dz  \leq N(\alpha,d).
\end{gathered}
\end{equation}
Therefore, we have \eqref{bound_1_ineq} for $k=0$. One can handle the case $k=1$ in the same manner.

\textbf{Step 2.} $(k = 2)$.  By \eqref{scaling_qxx}, we only need to consider $\delta = 2$.  Note that $f(t-4^{2/\alpha},\cdot) = 0$ for $t\leq 0$. By \eqref{eqn 05.16.1}, Lemma \ref{representation_whole_space} \eqref{representation2}, and change of variables, we have
\begin{equation*}
\begin{aligned}
\aint_{Q_{2}}|L_{2}f(s,y)|^{p_{0}}dyds 
&=\aint^{4^{2/\alpha}}_{-4^{2/\alpha}}\aint_{B_{3\delta/2}} |L_{2}f(s,y)|^{p_{0}}dyds\\
&\leq N(\alpha,d)\int_0^\infty\int_{\bR^d} |L_2(f(\cdot-4^{2/\alpha},\cdot))(s,y)|^{p_0}  dyds \\
& \leq N(\alpha,d,p_0) \int_0^\infty \int_{\bR^d} |f(s-4^{2/\alpha},y)|^{p_0}dyds \\
& \leq N(\alpha,d,p_0) \aint_{[-4^{2/\alpha},4^{2/\alpha}]\times B_{3d}}|f(s,y)|^{p_0}dyds. \\
\end{aligned}
\end{equation*}
The lemma is proved.
\end{proof}

\begin{lemma}\label{2_2_bound_2}
Let $p_{0}\in(1,\infty)$, $T<\infty$ and $f\in C^\infty_c(\R^{d+1})$. Assume $f=0$ for $|t|\geq (2\delta)^{2/\alpha}$. Then for  $(t,x)\in Q_{\delta}$,
\begin{equation} \label{bound_2_ineq}
\begin{gathered}
 \aint_{Q_{\delta}}| L_{k}f(s,y) |^{p_{0}}dyds \leq N_{k}(\alpha,d,p_0,T) \mathbb{M}|f|^{p_{0}}(t,x), \quad k=0,1, 
\\
\aint_{Q_{\delta}}| L_{2}f(s,y) |^{p_{0}}dyds \leq N_{2}(\alpha,d,p_0) \mathbb{M}|f|^{p_{0}}(t,x).
\end{gathered}
\end{equation}
\end{lemma}
\begin{proof} 
\textbf{Step 1.} $(k = 0,1)$.  Due to the similarity, we only consider $k=0$. Take $\zeta\in C_c^\infty(\mathbb{R}^d)$ such that $\zeta = 1$ in $B_{\delta d}$ and $\zeta = 0$ outside $B_{3\delta d/2}$. Since $L_{k}$ is linear, 
$$L_{k}f = L_{k}(f\zeta) + L_{k}(f (1- \zeta)).
$$
 For $L_{k}(f\zeta)$, we can apply  Lemma \ref{2_2_bound_1}, and therefore we may assume that $f(t,x) = 0$ if $x\in B_{\delta d}$.  Recall that 
\begin{equation*}
L_{0}f(s,y) = \int_{-\infty}^s\int_{\bR^d}1_{0<s-r<T}q(s-r,z)f(r,y-z)dzdr.
\end{equation*}
Let $(s,y)\in Q_\delta$ and $r\in(s-T,s)$. Since $|x-y|\leq \delta d$, $\rho>\delta d/2$ implies 
\begin{equation} \label{ball}
\quad B_{\rho}(y)\subset B_{\delta d+\rho}(x)\subset B_{3\rho}(x).
\end{equation}
Also, if $\rho\leq \delta d/2$ and $z\in B_{\rho}(0)$ then $f(r,y-z) = 0$  since $|y-z|\leq \delta d$. Now, observe that
\begin{equation} \label{during_proof_8}
\begin{aligned}
& \left|\int_{\bR^d}q(s-r,z)f(r,y-z)dz\right| \leq    N(\alpha,d,T) |s-r|^{\alpha/2-1}\int_{d\delta/2}^{\infty} \rho^{-d-2}\int_{B_{3\rho}(x)}|f(r,z)| dzd\rho.
\end{aligned}
\end{equation} 
Indeed, by Lemma \ref{useful_lemma} \eqref{bound_of_q_4}
\begin{eqnarray} 
  \nonumber
&& \bigg|\int_{\bR^d}q(s-r,z)f(r,y-z)dz\bigg|  \leq  N(\alpha,d) |s-r|^{\frac{3}{2}\alpha-1}\int_{|z|\geq d\delta/2}|z|^{-d-1}|f(r,y-z)|dz \\
&&= N(\alpha,d) T^{\alpha}  |s-r|^{\frac{1}{2}\alpha-1}\int_{|z|\geq d\delta/2}|z|^{-d-1}|f(r,y-z)|dz.   \label{eqn 10.22.1}
\end{eqnarray} 
By Lemma \ref{useful_lemma} \eqref{int_by_part}, \eqref{ball}, and $|s-r|<T$, we have
\begin{equation*} 
\begin{aligned}
&|s-r|^{\alpha/2-1}\int_{|z|\geq d\delta/2}|z|^{-d-1}|f(r,y-z)| dz
\\
&\quad=  N(d) |s-r|^{\alpha/2-1}\int_{d\delta/2}^{\infty} \rho^{-d-2}\int_{|z|\leq\rho}|f(r,y-z)| dzd\rho \\
&\quad\leq  N(d) |s-r|^{\alpha/2-1}\int_{d\delta/2}^{\infty} \rho^{-d-2}\int_{B_{3\rho}(x)}|f(r,z)| dzd\rho.
\end{aligned}
\end{equation*}
Thus, we have \eqref{during_proof_8}. By \eqref{during_proof_8} and H\"older's inequality
\begin{equation} \label{during_proof_2}
\begin{aligned}
&\aint_{Q_{\delta}}   |L_{0}f(s,y)|^{p_{0}} dyds  \\
& \leq N\aint_{Q_\delta}\left|  \int^s_{-(2\delta)^{2/\alpha}}|s-r|^{\alpha/2-1} \int_{\delta d/2}^\infty \rho^{-d-2}\int_{B_{3\rho}(x)} |f(r,z)| dz d\rho dr  \right|^{p_0}dyds \\
& \leq N\aint_{Q_\delta} I_0^{p_0-1} \int^s_{-(2\delta)^{2/\alpha}} |s-r|^{\alpha/2-1}\int_{\delta d/2}^\infty \rho^{-d-2}\int_{B_{3\rho}(x)} |f(r,z)|^{p_0} dz d\rho dr   dyds,
\end{aligned}
\end{equation}
where $N=N(\alpha,d,T)$ and 
\begin{equation*}
I_0 = I_0(\delta,s) := \int_{-(2\delta)^{2/\alpha}}^s |s-r|^{\alpha/2-1}dr \int_{\delta d/2}^\infty\rho^{-2}d\rho.
\end{equation*}
By a change of variables, we have
\begin{equation*}
I_0 \leq \int_0^{(2\delta)^{2/\alpha}}|r|^{\alpha/2-1}dr\int_{\delta d/2}^\infty\rho^{-2}d\rho \leq N(\alpha,d)\delta\delta^{-1} = N(\alpha,d).
\end{equation*}
By \eqref{during_proof_2} and Fubini's theorem,
\begin{equation} \label{eqn 05.17.1}
\begin{aligned}
&\aint_{Q_{\delta}}   |L_{0}f(s,y)|^{p_{0}} dyds  \\
&\leq N \aint_{Q_\delta} \int^s_{-(2\delta)^{2/\alpha}} |s-r|^{\alpha/2-1} \int_{\delta d/2}^\infty \rho^{-d-2}\int_{B_{3\rho}(x)} |f(r,z)|^{p_0} dz d\rho dr   dyds \\
&\quad\leq N\delta^{-2/\alpha}\int_{\delta d/2}^\infty\rho^{-d-2} \int_{-(2\delta)^{2/\alpha}}^0\int_r^0 |s-r|^{\alpha/2-1} ds \int_{B_{3\rho}(x)} |f(r,z)|^{p_0} dz   dr d\rho \\
&\quad\leq N\delta^{-2/\alpha}\int_{\delta d/2}^\infty\rho^{-d-2} \int_{-(2\delta)^{2/\alpha}}^0|r|^{\alpha/2} \int_{B_{3\rho}(x)} |f(r,z)|^{p_0} dz   dr d\rho \\
&\quad\leq N\delta^{1-2/\alpha}\int_{\delta d/2}^\infty\rho^{-d-2} \int_{-(2\delta)^{2/\alpha}}^0 \int_{B_{3\rho}(x)} |f(r,z)|^{p_0} dz   dr d\rho \\
&\quad \leq N \bM |f|^{p_0}(t,x),
\end{aligned}
\end{equation}
where $N=N(\alpha,d,p_0, T)$. Therefore, we have \eqref{bound_2_ineq}.

\textbf{Step 2.} $(k=2)$. By \eqref{scaling_qxx}, we may assume $\delta = 2$. Observe that
\begin{equation*}
\begin{aligned}
&\aint_{Q_{2}}|L_{2}f(s,y)|^{p_{0}}dyds \leq\int_{Q_2}\left| \int_{-\infty}^s\int_{\bR^d} D_{ij}q(s-r,z)f(r,y-z) dzdr \right|^{p_0}dyds
\end{aligned}
\end{equation*}
By Lemma \ref{useful_lemma} \eqref{bound_of_q_4},
$$
\bigg|\int_{\bR^d}D^2q(s-r,z)f(r,y-z)dz\bigg| \leq N(\alpha,d) |s-r|^{\alpha/2-1}\int_{|z|\geq d\delta/2}|z|^{-d-1}|f(r,y-z)|dz.
$$
Unlike in \eqref{eqn 10.22.1}, the constant $N$ above is independent of $T$. Thus, as  in \eqref{during_proof_8}, we get that for $0>s>r>-4^{2/\alpha}$ and $y\in[-1,1]^d$,
\begin{equation*}
\left| \int_{\bR^d}D_{ij}q(s-r,z)f(r,y-z)dz \right|\leq  N(\alpha,d) |s-r|^{\alpha/2-1}\int_{\delta/2}^\infty\rho^{-d-2} \int_{B_{3d\rho}(x)} |f(r,z)| dz.
\end{equation*}
This is because, By  \eqref{during_proof_2}, \eqref{eqn 05.17.1}, we have \eqref{bound_2_ineq}. The lemma is proved.
\end{proof}

\begin{lemma}\label{lem 05.17.1}
Let $p_0\in(1,\infty)$, $T<\infty$ and $f\in C_c^\infty(\R^{d+1})$. Suppose that $f=0$ outside of $(-\infty,-(\frac{3\delta}{2})^{2/\alpha})\times B_{3\delta d/2}$. Then for $(t,x)\in Q_{\delta}$,
\begin{equation} \label{bound_3_ineq}
\begin{gathered}
 \aint_{Q_{\delta}}| L_{k}f(s,y) |^{p_{0}}dyds \leq N_{k}(\alpha,d,p_0,T) \mathbb{M}|f|^{p_{0}}(t,x), \quad k=0,1
\\
\aint_{Q_{\delta}}| L_{2}f(s,y) |^{p_{0}}dyds \leq N_{2}(\alpha,d,p_0) \mathbb{M}|f|^{p_{0}}(t,x).
\end{gathered}
\end{equation}
\end{lemma}
\begin{proof}
\textbf{Step 1.} $(k = 0,1)$. As in the proof of Lemma \ref{2_2_bound_2}, we only consider the case $k = 0$. Let $(s,y)\in Q_{\delta}$. Notice that if $|z|\geq 2\delta d$ then $f(r,y-z) = 0$ since $|y-z|\geq |z| - |y|\geq3\delta d/2$. Also, $f(r,\cdot)=0$ if $r\geq-(3\delta/2)^{2/\alpha}$. Thus by H\"older's inequality,
\begin{equation}
\begin{aligned}
|L_0f(s,y)|^{p_0}
& \leq I_0^{p_0-1}\int_{-\infty}^{-(\frac{3\delta}{2})^{2/\alpha}} \int_{|z|\leq 2\delta d}1_{0<s-r<T}|q(s-r,z)||f(r,y)|^{p_0}dzdr,
\end{aligned}
\end{equation}
where
\begin{equation*}
I_0 = I_0(\delta,s) = \int_{-\infty}^{-(3\delta/2)^{2/\alpha}}1_{0<s-r<T}\int_{|z|\leq 2\delta d}|q(s-r,z)|dzdr.
\end{equation*}
Observe that for $r<s<r+T$, by Lemma \ref{useful_lemma} \eqref{bound_of_q_4},
\begin{equation} \label{during_proof_4}
\begin{aligned}
\int_{|z|\leq 2\delta d}|q(s-r,z)|dz &\leq N|s-r|^{\alpha/2-1}\int_{|z|\leq 2\delta d} |z|^{-d+1}dz \\
&\leq N(\alpha,d,T) |s-r|^{-\alpha/2-1} \delta.
\end{aligned}
\end{equation}
Also, for $s\in(-\delta^{2/\alpha},0)$ and $r\in(-\infty,-(3\delta/2)^{2/\alpha})$, we have
\begin{equation} \label{comparable}
|r-s|\leq 2|r|\leq N(\alpha)|r-s|.
\end{equation}
Thus, by \eqref{during_proof_4} and \eqref{comparable}
\begin{equation} \label{during_proof_9}
\begin{aligned}
I_0(\delta,s) & \leq N(\alpha,d,T) \,\,\delta \int_{-\infty}^{-(3\delta/2)^{2/\alpha}} |r|^{-\alpha/2-1}dr \leq N(\alpha,d,T).
\end{aligned}
\end{equation}
Therefore, by Fubini's theorem, change of variables, \eqref{during_proof_4}, and \eqref{comparable}, we have
\begin{equation*}
\begin{aligned}
&\int_{Q_\delta} |L_0f(s,y)|^{p_0} dyds\\
&\quad\leq N \int_{Q_\delta}\int_{s-T}^{-(3\delta/2)^{2/\alpha}}\int_{|z|\leq 2\delta d}|q(s-r,z)||f(r,y-z)|^{p_0}dzdrdyds \\
&\quad\leq N \int_{-\delta^{2/\alpha}}^0\int_{s-T}^{-(3\delta/2)^{2/\alpha}}\int_{|z|\leq 2\delta d}|q(s-r,z)|dz\int_{|y|\leq5\delta d/2}|f(r,y)|^{p_0}dydrds \\
&\quad\leq N\delta\int_{-\delta^{2/\alpha}}^0\int_{-\infty}^{-(3\delta/2)^{2/\alpha}}|s-r|^{-\alpha/2-1}\int_{|y|\leq5\delta d/2}|f(r,y)|^{p_0}dydrds\\
&\quad\leq N\delta^{1+2/\alpha}\int_{-\infty}^{-(3\delta/2)^{2/\alpha}}|r|^{-\alpha/2-1}\int_{|y|\leq 5\delta d/2}|f(r,y)|^{p_0}dydr,
\end{aligned}
\end{equation*}
where the constants $N$ above depend only on $\alpha,d,p_0$ and $T$.
By integration by parts, we have
\begin{equation*}
\begin{aligned}
\int_{Q_\delta} |L_0f(s,y)|^{p_0}dyds & \leq N \delta^{2/\alpha+1}\int_{-\infty}^{-(3\delta/2)^{2/\alpha}}|r|^{-\alpha/2-2}\int_{r}^0\int_{|y|\leq 5\delta d/2}|f(\sigma,y)|^{p_0}dyd\sigma dr\\
& \leq N\delta^{d+2/\alpha+1}\int_{-\infty}^{-(3\delta/2)^{2/\alpha}}|r|^{-\alpha/2-1} dr\bM|f|^{p_0}(t,x)\\
& \leq N(\alpha,d,p_{0},T) \delta^{d+2/\alpha} \bM |f|^{p_0}(t,x).
\end{aligned}
\end{equation*}
Therefore, we have \eqref{bound_3_ineq}.

\textbf{Step 2.} $(k=2)$. By \eqref{scaling_qxx}, we may assume $\delta = 2$. Take $\beta\in (1,\frac{d}{d-1})$ (here, $\frac{1}{0}:=\infty$). By Minkowski's inequality
\begin{equation} \label{during_proof_6}
\begin{aligned}
&\aint_{Q_2} |L_2f(s,y)|^{p_0} dyds \\
& =N(\alpha,d) \int_{-2^{2/\alpha}}^0\int_{[-1,1]^d} \bigg| \int_{-\infty}^{-3^{2/\alpha}}\int_{|z|\leq 4 d}|D_{ij}q(s-r,z)f(r,y-z)|dzdr \bigg|^{p_0}dyds \\
&\leq N(\alpha,d) \int_{-2^{2/\alpha}}^0 \left( \int_{-\infty}^{-3^{2/\alpha}}\int_{|z|\leq 4 d}|D_{ij}q(s-r,z)|dz\left(\int_{|y|\leq 5 d} |f(r,y)|^{p_0}dy\right)^{\frac{1}{p_0}}dr \right)^{p_0}ds. \\
\end{aligned}
\end{equation}
By \eqref{scaling_of_q}, Jensen's inequality, change of variables, \eqref{bounds_of_q_2} and \eqref{bounds_of_q_1}, we have
\begin{equation*}
\begin{aligned}
\int_{|z|\leq 4 d} |D_{ij}q(s-r,z)|dz 
&\leq|s-r|^{-(1-\frac{1}{\beta})\frac{\alpha d}{2}-1}\left(\int_{\bR^d}|D_{ij}q(1,z)|^{\beta}dz\right)^{1/\beta}\\
& \leq N(\alpha,\beta,d) |s-r|^{-(1-\frac{1}{\beta})\frac{\alpha d}{2}-1}.
\end{aligned}
\end{equation*}
Also note that one can replace $|s-r|$ with $|r|$ if $r<-3^{2/\alpha}$, and $-2^{2/\alpha}<s<0$. Therefore, 
\begin{equation*}
\begin{aligned}
& \aint_{Q_2} |L_2f(s,y)|^{p_0} dyds \\
&\leq N(\alpha,\beta,d) \int^0_{-2^{2/\alpha}} (\int^{-3^{\alpha/2}}_{-\infty}|r|^{-(1-\frac{1}{\beta})\frac{\alpha d}{2}-1} (\int_{|y|\leq5d}|f(r,y)|^{p_0}dy)^{1/p_0} dr)^{p_0}ds.
\end{aligned}
\end{equation*}
By \eqref{comparable}, H\"older's inequality, and integration by parts, we have
\begin{equation}
\begin{aligned}
& \int^0_{-2^{2/\alpha}} ( \int_{-\infty}^{-3^{2/\alpha}}|r|^{-(1-\frac{1}{\beta})\frac{\alpha d}{2}-1}(\int_{|y|\leq5d}|f(r,y)|^{p_0}dy)^{1/p_0}dr )^{p_0}ds\\
&\quad = N(\alpha) \left( \int_{-\infty}^{-3^{2/\alpha}}|r|^{-(1-\frac{1}{\beta})\frac{\alpha d}{2}-1}(\int_{|y|\leq5d}|f(r,y)|^{p_0}dy)^{1/p_0}dr \right)^{p_0} \\
&\quad\leq N(\alpha,\beta,d,p_0) \int_{-\infty}^{-3^{2/\alpha}}|r|^{-(1-\frac{1}{\beta})\frac{\alpha d}{2}-1}\int_{|y|\leq5d}|f(r,y)|^{p_0}dydr \\
&\quad\leq N(\alpha,\beta,d,p_0) \int_{-\infty}^{-3^{2/\alpha}} |r|^{-(1-\frac{1}{\beta})\frac{\alpha d}{2}-2} \int_{r}^0\int_{|y|\leq5d}|f(s,y)|^{p_0}dyds dr\\
&\quad\leq N(\alpha,\beta,d,p_0) \int_{-\infty}^{-3^{2/\alpha}} |r|^{-(1-\frac{1}{\beta})\frac{\alpha d}{2}-1}dr\bM |f|^{p_0}(t,x)\\
&\quad\leq N(\alpha,d,p_0) \bM |f|^{p_0}(t,x).
\end{aligned}
\end{equation}
The lemma is proved.
\end{proof}

\begin{lemma}\label{2_2_bound_3}
Let $p_0\in(1,\infty)$, $T<\infty$, and $f\in C_c^\infty(\R^{d+1})$. Suppose that $f=0$ outside of $(-\infty,-(\frac{3\delta}{2})^{2/\alpha})\times B_{\delta d}^c$. Then for $(t,x)\in Q_{\delta}$,
\begin{equation} \label{bound_4_ineq}
\begin{gathered}
\aint_{Q_{\delta}} \aint_{Q_{\delta}} | L_{k}f(s,y) - L_{k}f(r,z) |^{p_{0}}dydsdzdr \leq N(\alpha,d,p_0,T) \mathbb{M}|f|^{p_{0}}(t,x), \quad (k=0,1)
\\
\aint_{Q_{\delta}} \aint_{Q_{\delta}} | L_{2}f(s,y) - L_{2}f(r,z) |^{p_{0}}dydsdzdr \leq N(\alpha,d,p_0) \mathbb{M}|f|^{p_{0}}(t,x).
\end{gathered}
\end{equation}
\end{lemma}
\begin{proof}
\textbf{Step 1.} $(k=0,1)$. Again, due to the similarity we only consider the case $k=0$. 

Obviously, to prove the claim  it suffices to show that
\begin{equation*}
\aint_{Q_\delta}|L_k f|^{p_0} \leq N(\alpha,d,p_0,T)  \bM|f|^{p_0}(t,x), \quad \forall \, (t,x)\in Q_{\delta}.
\end{equation*}
 Note that $f(r,y-z) = 0$ for $r\geq -(3\delta/2)^{2/\alpha}$ or $(y,z)\in [-\delta/2,\delta/2]^d\times B_{\delta d/2}$. For $(s,y)\in Q_\delta$, by H\"older's inequality
\begin{equation} \label{third_partition_ineq}
\begin{aligned}
| L_0f(s,y) |^{p_0}  &= \left|\int_{-\infty}^s\int_{\bR^d}  1_{0<s-r<T}q(s-r,z)f(r,y-z) dzdr\right|^{p_0}  \\
&\leq I_0^{p_0-1}\int_{-\infty}^s\int_{|z|\geq \frac{\delta d}{2}} 1_{0<s-r<T}| q(s-r,z)||f(r,y-z)|^{p_0} dzdr, \\
\end{aligned}
\end{equation}
where
\begin{equation*}
I_0 = I_0(s, \delta) = \int_{-\infty}^s 1_{0<s-r<T}\int_{\bR^d} |q(s-r,z)|dzdr.
\end{equation*}
By \eqref{during_proof_11}, we have
\begin{equation*}
\begin{aligned}
I_0 \leq \int_0^T r^{\alpha -1}dr\int_{\bR^d}|q(1,z)|dz\leq N(\alpha,d,T).
\end{aligned}
\end{equation*}
Thus, 
\begin{equation} \label{during_proof_12}
\begin{aligned}
&\int_{Q_\delta}| L_0f(s,y)|^{p_0}dyds\\
&\leq N(\alpha,d,T) \int_{Q_\delta}\int_{-\infty}^{-\left(\frac{3\delta}{2}\right)^{\frac{2}{\alpha}}} 1_{0<s-r<T}\int_{|z|\geq\frac{\delta d}{2}}|q(s-r,z)||f(r,y-z)|^{p_0}dzdrdyds. \\
\end{aligned}
\end{equation}
Note that for $x,y\in [-\delta/2,\delta/2]^d$ and $\rho>\delta d/2$, we have 
\begin{equation} \label{ball2}
B_\rho(y)\subset B_{(2/\sqrt{d}+1)\rho}(x)\subset B_{3\rho}(x).
\end{equation}
To proceed further, we consider  two differenent cases. 

\textbf{Case 1.} ($\delta \geq 1$). Let $\ep\in (0,1)$. By Lemma \ref{useful_lemma} \eqref{bound_of_q_4}, we have
\begin{equation} \label{during_proof_10}
|q(s-r,z)| \leq N(\alpha,d, \varepsilon) |s-r|^{(1+\frac{\ep}{2})\alpha-1}|z|^{-d-\ep}.
\end{equation}
By \eqref{during_proof_12}, \eqref{during_proof_10}, Lemma \ref{useful_lemma} \eqref{int_by_part}, \eqref{ball2}, Fubini's theorem, \eqref{comparable}, and integration by parts, we have
\begin{equation*}
\begin{aligned}
& \int_{Q_\delta}| L_0f(s,y)|^{p_0}dyds \\
&\quad\leq N \int_{Q_\delta}\int^{-\left(\frac{3\delta}{2}\right)^{2/\alpha}}_{s-T}|s-r|^{(1+\frac{\ep}{2})\alpha - 1}\int_{|z|\geq\frac{\delta d}{2}}|z|^{-d-\ep}|f(r,y-z)|^{p_0}dzdrdyds \\
&\quad\leq N \int_{Q_\delta}\int_{s-T}^{-\left(\frac{3\delta}{2}\right)^{2/\alpha}}|s-r|^{(1+\frac{\ep}{2})\alpha - 1}\int^\infty_{\frac{\delta d}{2}}\rho^{-\ep-1}\aint_{B_{3\rho}(x)}|f(r,z)|^{p_0}dzd\rho drdyds \\
&\quad\leq N\delta^d\int^\infty_{\frac{\delta d}{2}}\rho^{-\ep-1}\int_{-\delta^{2/\alpha}}^0 \int^{-\left(\frac{3\delta}{2}\right)^{2/\alpha}}_{-\infty} |s-r|^{ -\frac{\alpha\ep}{2}- 1}\aint_{B_{3\rho}(x)}|f(r,z)|^{p_0}dz dr ds d\rho \\
&\quad\leq N\delta^d\int^\infty_{\frac{\delta d}{2}}\rho^{-\ep-1}\int_{-\delta^{2/\alpha}}^0 \int^{-\left(\frac{3\delta}{2}\right)^{2/\alpha}}_{-\infty} |r|^{ -\frac{\alpha\ep}{2}- 2}\int_r^0\aint_{B_{3\rho}(x)}|f(\sigma,z)|^{p_0}dzd\sigma dr ds d\rho \\
&\quad\leq N\delta^{d+2/\alpha-2\ep}\bM|f|^{p_0}(t,x),
\end{aligned}
\end{equation*}
where $N=N(\alpha,d,p_0,\varepsilon,T)$. Therefore, we have (recall that $\delta \geq 1$)
\begin{equation*}
\aint_{Q_\delta}| L_0f(s,y)|^{p_0}dyds \leq \delta^{-2\ep}N \bM|f|^{p_0}(t,x) \leq N \bM|f|^{p_0}(t,x).
\end{equation*}

\textbf{Case 2.} ($\delta <1$). Let $\ep\in(0,1)$. By Lemma \ref{useful_lemma} \eqref{bound_of_q_4}, we have
\begin{equation} \label{during_proof_13}
\begin{aligned}
|q(s,z)|  &\leq (1_{|z|\leq d/2}+1_{|z|\geq d/2})|q(s,z)| \\
&\leq N(\alpha,d,\varepsilon) \left( 1_{|z|\leq \frac{d}{2}}|s|^{(1-\frac{\ep}{2})\alpha-1}|z|^{-d+\ep}+1_{|z|\geq \frac{d}{2}}|s|^{(1+\frac{\ep}{2})\alpha-1}|z|^{-d-\ep}\right).
\end{aligned}
\end{equation}
Also, for $x,y\in [-\delta/2,\delta/2]^d$,
\begin{equation} \label{ball3}
B_{d/2}(y)\subset B_{\sqrt{d}\delta+d/2}(x)\subset B_{3d/2}(x).
\end{equation}
Then by \eqref{during_proof_12}, \eqref{during_proof_13},
\begin{equation*}
\begin{aligned}
& \int_{Q_\delta}| L_0f(s,y)|^{p_0}dyds \\
& \leq N\int_{Q_\delta}\int_{-\infty}^{-\left(\frac{3\delta}{2}\right)^{2/\alpha}} \int_{|z|\geq\frac{\delta d}{2}}1_{0<s-r<T}|q(s-r,z)||f(r,y-z)|^{p_0}dzdrdyds \\
&\leq N\int_{Q_\delta}\int^{-\left(\frac{3\delta}{2}\right)^{2/\alpha}}_{s-T}|s-r|^{(1-\frac{\ep}{2})\alpha-1}\int_{\frac{\delta d}{2}\leq|z|\leq\frac{d}{2}}|z|^{-d+\ep}|f(r,y-z)|^{p_0}dzdrdyds\\
&\quad\quad+N\int_{Q_\delta}\int^{-\left(\frac{3\delta}{2}\right)^{2/\alpha}}_{s-T}|s-r|^{(1+\frac{\ep}{2})\alpha-1}\int_{|z|\geq\frac{d}{2}}|z|^{-d-\ep}|f(r,y-z)|^{p_0}dzdrdyds \\
&=: N(I_1+I_2),
\end{aligned}
\end{equation*}
where $N=N(\alpha,d,p_{0},\varepsilon,T)$. 
Therefore, it suffices to show that
\begin{equation*}
I_1+I_2\leq N(\alpha,d,p_0,T) \delta^{d+2/\alpha}\bM|f|^{p_0}(t,x).
\end{equation*}
Consider $I_1$. Note that $|y-z|\leq \delta d$ if $z\in B_{\delta d/2}$, and $y\in[-\delta/2,\delta/2]$. Using \eqref{int_by_part1}, \eqref{ball3}, \eqref{ball2}, and the assumption on the support of $f$, observe that
\begin{equation*}
\begin{aligned}
&\int_{\frac{\delta d}{2}\leq|z|\leq\frac{d}{2}}|z|^{-d+\ep}|f(r,y-z)|^{p_0}dz \\
&=(d/2)^{-d+\varepsilon}\int_{B_{d/2}(0)}|f(r,y-z)|^{p_0}dz
-(\delta d/2)^{-d+\varepsilon} \int_{B_{\delta d/2}(0)}|f(r,y-z)|^{p_0}dz\\
&\quad + (d-\varepsilon) \int^{\frac{d}{2}}_{\frac{\delta d}{2}}\rho^{-d+\ep-1}\int_{B_{\rho}(0)}|f(r,y-z)|^{p_0}dzd\rho\\
& \quad\leq N(d) \left[\int_{B_{3d/2}(x)}|f(r,z)|^{p_0}dz +\int^{\frac{d}{2}}_{\frac{\delta d}{2}}\rho^{-d+\ep-1}\int_{B_{3\rho}(x)}|f(r,z)|^{p_0}dzd\rho\right].
\end{aligned}
\end{equation*}
Therefore,  
\begin{equation*}
\begin{aligned}
I_1 &\leq N \int_{Q_\delta}\int^{-\left(\frac{3\delta}{2}\right)^{2/\alpha}}_{s-T}|s-r|^{(1-\frac{\ep}{2})\alpha-1} \Big[\int_{B_{3d/2}(x)}|f(r,z)|^{p_0}dz \\
&\quad\quad\quad\quad\quad+\int^{\frac{d}{2}}_{\frac{\delta d}{2}}\rho^{-d+\ep-1}\int_{B_{3\rho}(x)}|f(r,z)|^{p_0}dzd\rho\Big] dr dyds\leq N(I_{11}+I_{12}),
\end{aligned}
\end{equation*}
where
\begin{equation*}
\begin{aligned}
I_{11}&:=\delta^d\int^0_{-\delta^{2/\alpha}}\int^{-\left(\frac{3\delta}{2}\right)^{2/\alpha}}_{s-T}|s-r|^{(1-\frac{\ep}{2})\alpha - 1}\int_{B_{3d/2}(x)}|f(r,z)|^{p_0}dzdrds,\\
I_{12}&:=\delta^d\int^0_{-\delta^{2/\alpha}}\int^{-\left(\frac{3\delta}{2}\right)^{2/\alpha}}_{s-T}|s-r|^{(1-\frac{\ep}{2})\alpha - 1}\int^{\frac{d}{2}}_{\frac{\delta d}{2}}\rho^{-d+\ep-1}\int_{B_{3\rho}(x)}|f(r,z)|^{p_0}dzd\rho drds.
\end{aligned}
\end{equation*}
  For $I_{11}$, by integration by parts with respect to $r$, and \eqref{comparable}, we have
\begin{equation*}
\begin{aligned}
I_{11} &:=  \delta^d\int^0_{-\delta^{2/\alpha}} \left(\int^{-\left(\frac{3\delta}{2}\right)^{2/\alpha}}_{s-T}|s-r|^{(1-\frac{\ep}{2})\alpha - 1}\int_{B_{3d/2}(x)}|f(r,z)|^{p_0}dz  dr \right) ds\\
&\leq N\delta^d\int^0_{-\delta^{2/\alpha}} \Bigg[ T^{\left( 1-\frac{\ep}{2} \right)\alpha-1}\int_{s-T}^0 \int_{B_{3d/2}(x)}|f(\sigma,z)|^{p_0}dzd\sigma \\
&\quad\quad\quad\quad\quad+\int^{-\left(\frac{3\delta}{2}\right)^{2/\alpha}}_{s-T}|s-r|^{(1-\frac{\ep}{2})\alpha - 2}\int_r^0\int_{B_{3d/2}(x)}|f(\sigma,z)|^{p_0}dzd\sigma dr \Bigg]ds\\
&\leq N\delta^d\int^0_{-\delta^{2/\alpha}}\Bigg[ (1+T)^{(1-\varepsilon/2)\alpha}+\int^{s}_{s-T}|s-r|^{(1-\frac{\ep}{2})\alpha - 1} dr \Bigg]ds\bM|f|^{p_0}(t,x)\\
&\leq N\delta^{d+2/\alpha}\bM|f|^{p_0}(t,x).
\end{aligned}
\end{equation*}
 For $I_{12}$, by Fubuni's theorem, integration by parts with respect to $r$, and \eqref{comparable}, we have (recall $\delta<1$)
\begin{equation*}
\begin{aligned}
I_{12}&=\delta^d  \int^{\frac{d}{2}}_{\frac{\delta d}{2}} \int^0_{-\delta^{2/\alpha}}\rho^{-d+\ep-1}\left(\int^{-\left(\frac{3\delta}{2}\right)^{2/\alpha}}_{s-T}|s-r|^{(1-\frac{\ep}{2})\alpha - 1}\int_{B_{3\rho}(x)}|f(r,z)|^{p_0}dz dr \right) ds  d\rho\\
&\leq N\delta^d\int^{\frac{d}{2}}_{\frac{\delta d}{2}}\rho^{-d+\ep-1}\int^0_{-\delta^{2/\alpha}} \Bigg[T^{\left( 1-\frac{\ep}{2} \right)\alpha-1}\int_{s-T}^0\int_{B_{3\rho}(x)}|f(\sigma,z)|^{p_0}dzd\sigma\\
&\quad\quad\quad\quad\quad\quad+\int^{-\left(\frac{3\delta}{2}\right)^{2/\alpha}}_{s-T}|s-r|^{(1-\frac{\ep}{2})\alpha - 2} \int_r^0\int_{B_{3\rho}(x)}|f(\sigma,z)|^{p_0}dzd\sigma dr\Bigg] ds d\rho \\
&\leq N\delta^d\int_{\frac{\delta d}{2}}^{\frac{d}{2}}\rho^{\ep-1}\int^0_{-\delta^{2/\alpha}}\left[(1+T)^{(1-\varepsilon/2)\alpha}+\int_{s-T}^{s}|s-r|^{(1-\frac{\ep}{2})\alpha-1}dr\right]dsd\rho \bM|f|^{p_0}(t,x)\\
&\leq N\delta^d\int^{\frac{d}{2}}_0\rho^{\ep-1}d\rho\int^0_{-\delta^{2/\alpha}} ds \,\bM|f|^{p_0}(t,x) \\
&\leq N \delta^{d+2/\alpha}\bM|f|^{p_0}(t,x).
\end{aligned}
\end{equation*}

Now we consider $I_2$. Again, by \eqref{int_by_part1},
\begin{eqnarray*}
&&\int_{ |z|\geq \frac{d}{2}}|z|^{-d-\ep}|f(r,y-z)|^{p_0}dz\\
&=&-( d/2)^{-d-\varepsilon} \int_{B_{d/2}(0)}|f(r,y-z)|^{p_0}dz
 + (d+\varepsilon) \int^{\infty}_{\frac{ d}{2}}\rho^{-d-\ep-1}\int_{B_{\rho}(0)}|f(r,y-z)|^{p_0}dzd\rho\\
&\leq&  2d \int^{\infty}_{\frac{ d}{2}}\rho^{-d-\ep-1}\int_{B_{\rho}(0)}|f(r,y-z)|^{p_0}dzd\rho
 \end{eqnarray*}
Thus, by
\eqref{comparable}, Fubini's theorem,  and integration by parts with respect to $r$,
\begin{equation*}
\begin{aligned}
I_2 &:= \int_{Q_\delta}\int^{-\left(\frac{3\delta}{2}\right)^{2/\alpha}}_{s-T}|s-r|^{(1+\frac{\ep}{2})\alpha-1}\int_{ |z|\geq \frac{d}{2}}|z|^{-d-\ep}|f(r,y-z)|^{p_0}dzdrdyds \\
&\leq N\delta^d\int^\infty_{\frac{d}{2}}\rho^{-d-\ep-1}\int^0_{-\delta^{2/\alpha}} \left(\int^{-\left(\frac{3\delta}{2}\right)^{2/\alpha}}_{s-T}|s-r|^{(1+\frac{\ep}{2})\alpha - 1}\int_{B_{3\rho}(x)}|f(r,z)|^{p_0}dz dr \right)ds d\rho \\
&\leq N\delta^d\int_{\frac{d}{2}}^\infty\rho^{-d-\ep-1}\int^0_{-\delta^{2/\alpha}} \Bigg[T^{(1+\frac{\varepsilon}{2})\alpha-1}\int_{s-T}^0\int_{B_{3\rho}(x)}|f(\sigma,z)|^{p_0}dzd\sigma\\
&\quad\quad\quad\quad\quad\quad\quad\quad\quad\quad+\int^{s}_{s-T}|s-r|^{(1+\frac{\ep}{2})\alpha - 2} \int_r^0\int_{B_{3\rho}(x)}|f(\sigma,z)|^{p_0}dzd\sigma dr\Bigg] ds d\rho \\
&\leq N\delta^d\int_{\frac{d}{2}}^\infty\rho^{-\ep-1}d\rho\int^0_{-\delta^{2/\alpha}} (T+1)^{(1+\varepsilon/2)\alpha} ds \,\bM|f|^{p_0}(t,x) \\
&\leq N(\alpha,d,p_{0},T) \delta^{d+2/\alpha}\bM|f|^{p_0}(t,x).
\end{aligned}
\end{equation*}
 Therefore, \eqref{bound_4_ineq} is proved for $k = 0$. Similarly, one can treat the case $k = 1$.

\textbf{Step 2.} $(k=2)$. Due to Poincar\'e's inequality, it is sufficient to show that
\begin{equation*}
\aint_{Q_\delta}( |\partial_s L_2f|^{p_0}+|D_yL_2f|^{p_0} )dyds \leq N(d,\alpha,p_0)\bM |f|^{p_0}(t,x), \quad \forall (t,x)\in Q_{\delta}.
\end{equation*}
Due to \eqref{scaling_qxx}, we may assume $\delta = 2$. Thus, we  will only prove
\begin{equation*}
\aint_{Q_2}(|\partial_sL_2f|^{p_0}+|D_yL_2f|^{p_0}) dyds \leq N\bM |f|^{p_0}(t,x).
\end{equation*}
Let  $\ep\in(0,(1-\frac{1}{p_0})\frac{\alpha}{2+\alpha+\alpha d})$. Observe that $f(r,y-z) = 0$ for $r\geq -3^{2/\alpha}$ or $(y,z)\in[-1,1]^d\times B_{d}$. Thus
\begin{equation*} 
\begin{aligned}
&\int_{Q_2} | D_yL_2f(s,y) |^{p_0} dyds \\
&\quad\leq \int^0_{-2^{2/\alpha}}\int_{[-1,1]^d} \left| \int^{-3^{2/\alpha}}_{-\infty}\int_{\mathbb{R}^d} |D_zq_{z^iz^j}(s-r,z)f(r,y-z)| dzdr\right|^{p_0} dyds \\
&\quad\leq\int^0_{-2^{2/\alpha}} \int_{[-1,1]^d}  I_3^{p_0-1}  \int^{-3^{2/\alpha}}_{-\infty}\int_{|z|\geq d} | D_zq_{z^iz^j}(s-r,z)|^{1+p_0\ep}|f(r,y-z)|^{p_0} dzdr dyds,
\end{aligned}
\end{equation*}
where 
\begin{equation*}
I_3 = I_3(s) = \int^{-3^{2/\alpha}}_{-\infty}\int_{\bR^d} | D_zq_{z^iz^j}(s-r,z)|^{1-p_0'\ep} dzdr,\quad 1/p_0+1/p_0' = 1.
\end{equation*}
By \eqref{scaling_of_q}, change of variables, and \eqref{comparable}, we have
\begin{equation}
\begin{aligned}
&I_3(s) = \int^{-3^{2/\alpha}}_{-\infty}\int_{\bR^d} | D_zq_{z^iz^j}(s-r,z)|^{1-p_0'\ep} dzdr \\
&\quad\leq \int^{-3^{2/\alpha}}_{-\infty}\int_{\mathbb{R}^d} |s-r|^{(-\frac{\alpha d}{2}-1-\frac{\alpha}{2})(1-p_0'\ep)}| D_zq_{z^iz^j}(1,(s-r)^{-\frac{\alpha}{2}}z) |^{1-p_0'\ep} dzdr \\
&\quad\leq \int_{-\infty}^{-3^{2/\alpha}}  |s-r|^{(-\frac{\alpha d}{2}-1-\frac{\alpha}{2})(1-p_0'\ep)+\frac{\alpha d}{2}}\int_{\mathbb{R}^d} | D_zq_{z^iz^j}(1,z)|^{1-p_0'\ep}dz dr \\
&\quad\leq\int_{-\infty}^{-3^{2/\alpha}}|r|^{-(\frac{\alpha d}{2}+1+\frac{\alpha}{2})(1-p_0'\ep)+\frac{\alpha d}{2}}dr\int_{\mathbb{R}^d}| D_zq_{z^iz^j}(1,z)|^{1-p_0'\ep}dz\\
\end{aligned}
\end{equation}
Since $ -(\frac{\alpha d}{2}+1+\frac{\alpha }{2})(1-p_0'\ep) +\frac{\alpha d}{2} <-1$, 
we have $I_3(s)\leq N(\alpha,p_0,\varepsilon)<\infty$.
Thus, by \eqref{scaling_of_q}, \eqref{bounds_of_q_3}, Lemma \ref{useful_lemma} \eqref{int_by_part}, Fubini's theorem, \eqref{comparable}, and integration by parts with respect to $r$,  we have
\begin{equation*}
\begin{aligned}
&\int_{Q_2} | D_yL_2f(s,y) |^{p_0} dyds \\
&\quad\leq N\int_{Q_2} \int^{-3^{2/\alpha}}_{-\infty}\int_{|z|\geq d} | D_zq_{z^iz^j}(s-r,z)|^{1+p_0\ep}|f(r,y-z)|^{p_0} dzdrdyds\\
&\quad\leq N\int_{Q_2}\int^{-3^{2/\alpha}}_{-\infty}|s-r|^{-(1+p_0\ep)}\int_{|z|\geq d}|z|^{-(d+1)(1+p_0\ep)}|f(r,y-z)|^{p_0}dzdrdyds\\
&\quad\leq N \int^0_{-2^{2/\alpha}}\int_{-\infty}^{-3^{2/\alpha}}|s-r|^{-(1+p_0\ep)}\int_{d}^\infty \rho^{-(d+1)(1+p_0\ep)-1}\int_{B_{3\rho}(x)}|f(r,z)|^{p_0}dzd\rho drds\\
\end{aligned}
\end{equation*}
\begin{equation*}
\begin{aligned}
&\leq N \int_{d}^\infty \rho^{-(d+1)(1+p_0\ep)-1}\left(\int_{-\infty}^{-3^{2/\alpha}}|r|^{-(1+p_0\ep)}\int_{B_{3\rho}(x)}|f(r,z)|^{p_0}dz dr \right) d\rho \\
&\leq N \int_{d}^\infty \rho^{-(d+1)(1+p_0\ep)-1}\int_{-\infty}^{-3^{2/\alpha}}|r|^{-(1+p_0\ep)-1}\int_r^0\int_{B_{3\rho}(x)}|f(\sigma,z)|^{p_0}dz d\sigma dr d\rho \\
&\leq N \int_{d}^\infty \rho^{-(d+1)p_0\ep-2} d\rho\int_{-\infty}^{-3^{2/\alpha}}|r|^{-(1+p_0\ep)} dr\bM|f|^{p_0}(t,x) \\
&\leq N(\alpha,d,p_0) \bM|f|^{p_0}(t,x).
\end{aligned}
\end{equation*}
Next, we show
\begin{equation*}
\int_{Q_2}|\partial_sLf|^{p_0} dyds \leq N(d,\alpha,p_0) \bM|f|^{p_0}(t,x), \quad \forall \, (t,x)\in Q_2.
\end{equation*}
Let $\ep\in(0,(1-\frac{1}{p_0})\frac{2}{\alpha d + 4})$. Recall that $f(r,y-z) = 0$ for $r\geq -3^{2/\alpha}$ or $(y,z)\in[-1,1]^d\times B_d$. Then, by H\"older inequality,
\begin{equation*}
\begin{aligned}
&\int_{Q_2}|\partial_sL_2f(s,y)|^{p_0}dyds \\
&\quad\leq \int^0_{-2^{\alpha/2}}\int_{[-1,1]^d} \left( \int^{-3^{\alpha/2}}_{-\infty}\int_{\mathbb{R}^d} |\partial_s q_{z^iz^j}(s-r,z)f(r,y-z)| dzdr\right)^{p_0} dyds \\
&\quad\leq\int^0_{-2^{2/\alpha}} \int_{[-1,1]^d}  I_4^{p_0-1}  \int^{-3^{2/\alpha}}_{-\infty}\int_{|z|\geq d} | \partial_sq_{z^iz^j}(s-r,z)|^{1+p_0\ep}|f(r,y-z)|^{p_0} dzdr dyds,
\end{aligned}
\end{equation*}
where 
$$I_4 = I_4(s) = \int^{-3^{2/\alpha}}_{-\infty}\int_{|z|\geq d} | \partial_sq_{z^iz^j}(s-r,z)|^{1-p_0'\ep} dzdr, \quad 1/p_0+1/p_0' = 1.
$$
Observe that by \eqref{scaling_of_q},
\begin{equation*}
|\partial_t q_{xx}(t,x)| \leq N(t^{-\frac{\alpha d}{2} -2}|q_{xx}(1,t^{-\frac{\alpha}{2}}x)|+t^{-\frac{\alpha d}{2} -2 -\frac{\alpha}{2}}|x||q_{xxx}(1,t^{-\frac{\alpha}{2}}x)|).
\end{equation*}
Then we have
\begin{equation*}
\begin{aligned}
I_4 &= \int^{-3^{2/\alpha}}_{-\infty}\int_{|z|\geq d} | \partial_sq_{z^iz^j}(s-r,z)|^{1-p_0'\ep} dzdr \\
&\leq N(\alpha,d,p_0,\varepsilon)\Bigg(\int^{-3^{2/\alpha}}_{-\infty}\int_{|z|\geq d} |(s-r)^{-\frac{\alpha d}{2}-2}q_{z^iz^j}(1,(s-r)^{-\frac{\alpha}{2}}z)|^{1-p_0'\ep}dzdr\\
&\quad\quad\quad\quad+\int^{-3^{2/\alpha}}_{-\infty}\int_{|z|\geq d} |(s-r)^{-\frac{\alpha d}{2}-2-\frac{\alpha}{2}}|z|q_{z^iz^jz^k}(1,(s-r)^{-\frac{\alpha}{2}}z)|^{1-p_0'\ep}dzdr\Bigg) \\
&=: N (I_{41} + I_{42}).
\end{aligned}
\end{equation*}
For $I_{41}$, by \eqref{scaling_of_q} and change of variables, 
\begin{equation*}
\begin{aligned}
I_{41} &= \int^{-3^{2/\alpha}}_{-\infty}\int_{|z|\geq d} |(s-r)^{-\frac{\alpha d}{2}-2}q_{z^iz^j}(1,(s-r)^{-\frac{\alpha}{2}}z)|^{1-p_0'\ep}dzdr\\
&\leq \int^{-3^{2/\alpha}}_{-\infty}|s-r|^{-(\frac{\alpha d}{2}+2)(1-p_0'\ep)+\frac{\alpha d}{2}}dr\int_{\mathbb{R}^d}|q_{z^iz^j}(1,z)|^{1-p_0'\ep} dz.
\end{aligned}
\end{equation*}
Since  $-(\frac{\alpha d}{2}+2)(1-p_0'\ep)+\frac{\alpha d}{2}<-1$ and \eqref{during_proof_11},  we have $I_{41}<N(\alpha,d,p_0)<\infty$.  Also, by the same arguments used for   $I_{41}$,
\begin{equation*}
\begin{aligned}
I_{42} &= \int^{-3^{2/\alpha}}_{-\infty}\int_{|z|\geq d} |(s-r)^{-\frac{\alpha d}{2}-2-\frac{\alpha}{2}}|z|q_{z^iz^jz^k}(1,(s-r)^{-\frac{\alpha}{2}}z)|^{1-p_0'\ep}dzdr \\
&\leq \int^{-3^{2/\alpha}}_{-\infty} |s-r|^{-(\frac{\alpha d}{2}+2)(1-p_0'\ep)+\frac{\alpha d}{2}}\int_{\mathbb{R}^d}(|z||q_{z^iz^jz^k}(1,z)|)^{1-p_0'\ep}dzdr, \\
\end{aligned}
\end{equation*}
and the last term is bounded by a constant which depends only on $\alpha,d,p_0$. 
Thus,
by \eqref{scaling_of_q}, \eqref{bounds_of_q_3}, Lemma \ref{useful_lemma} \eqref{int_by_part}, Fubini's theorem, \eqref{comparable}, and integration by parts again, we have
\begin{equation*}
\begin{aligned}
&\int_{Q_2}|\partial_sL_2f(s,y)|^{p_0}dyds \\
&\quad\leq N\int_{Q_2} \int^{-3^{2/\alpha}}_{-\infty}\int_{|z|\geq d} | \partial_sq_{z^iz^j}(s-r,z)|^{1+p_0\ep}|f(r,y-z)|^{p_0} dzdr dyds \\
&\quad\leq N\int_{Q_2}\int^{-3^{2/\alpha}}_{-\infty}  \int_{|z|\geq d}  \bigg(|(s-r)^{-\frac{\alpha d}{2}-2}q_{z^iz^j}(1,(s-r)^{-\frac{\alpha}{2}}z)|^{1+p_0\ep}\\
&\quad\quad\quad\quad+|(s-r)^{-\frac{\alpha d}{2}-2-\frac{\alpha}{2}}|z|q_{z^iz^jz^k}(1,(s-r)^{-\frac{\alpha}{2}}z)|^{1+p_0\ep}\bigg) |f(r,y-z)|^{p_0}dzdrdyds \\
&\quad\leq N\int^0_{-2^{2/\alpha}} \int_{[-1,1]^d}\int^{-3^{2/\alpha}}_{-\infty}  \int_{|z|\geq d} |s-r|^{-2-2p_0\ep}|z|^{-d(1+p_0\ep)}|f(r,y-z)|^{p_0}dzdrdyds \\
&\quad\leq N\int^0_{-2^{2/\alpha}} \int_{[-1,1]^d}\int^{-3^{2/\alpha}}_{-\infty}   |r|^{-2-2p_0\ep}\int_{|z|\geq d}|z|^{-d(1+p_0\ep)}|f(r,y-z)|^{p_0}dzdrdyds \\
&\quad\leq N \int_{-\infty}^{-3^{2/\alpha}}|r|^{-2-p_0\ep}\int_{d}^\infty\rho^{-d(1+p_0\ep)-1}\int_{B_{3\rho}(x)}|f(r,z)|^{p_0}dzd\rho dr \\
&\quad\leq N \int_{d}^\infty \rho^{-d(1+p_0\ep)-1} \int_{-\infty}^{-3^{2/\alpha}}|r|^{-3-p_0\ep}\int_r^0\int_{B_{3\rho}(x)}|f(\sigma,z)|^{p_0}dzd\sigma dr d\rho\\
&\quad\leq N\int_d^\infty \rho^{-dp_0\ep-1}d\rho\int_{-\infty}^{-3^{2/\alpha}}|r|^{-2-p_0\ep}dr \bM|f|^{p_0}(t,x) \\
&\quad\leq N(\alpha,d,p_{0}) \bM|f|^{p_0}(t,x),
\end{aligned}
\end{equation*}
where the constants $N$ depend only on $\alpha,d,p_0$. 
The lemma is proved.

\end{proof}

For a measurable function $h(t,x)$ on $\mathbb{R}^{d+1}$, define the sharp function
\begin{equation*}
\begin{aligned}
h^{\#}(t,x) := \sup_{Q}\aint_{Q}|h(r,z)-h_{Q}|drdz,
\end{aligned}
\end{equation*}
where 
$$ h_{Q} = \aint_{Q}h(s,y)dyds=\frac{1}{|Q|}\int_{Q}h(s,y)dyds .
$$
The supremum is taken over all $Q\subset\mathbb{R}^{d+1}$ containing $(t,x)$ of the form 
\begin{equation*}
\begin{aligned}
Q &= Q_{\delta}(s,y) 
\\
&= \left(s-\frac{\delta^{\frac{2}{\alpha}}}{2},s+ \frac{\delta^{\frac{2}{\alpha}}}{2}\right)\times\left(y^1-\frac{\delta}{2},y^1+\frac{\delta}{2}\right)\times\cdots\times \left(y^d-\frac{\delta}{2},y^d+\frac{\delta}{2}\right)
\end{aligned}
\end{equation*}
with $\delta>0$. Observe that for any $c\in\R$, and $p_{0}\geq1$,
\begin{equation}\label{3.20}
\begin{aligned}
\aint_{Q}|h(s,y)-h_{Q}|^{p_{0}}dyds\leq 2^{p_{0}} \aint_{Q}|h(s,y)-c|^{p_{0}}dyds.
\end{aligned}
\end{equation}

Here is our sharp function estimate.

\begin{theorem} \label{thm 05.06.2}
Let $f\in C_c^\infty(\R^{d+1})$, and $p_0\in(1,\infty)$. Then we have
\begin{equation} \label{eqn 05.06.2}
\begin{gathered}
(L_{k}f)^{\#}\leq N_{k}(\alpha,d,p_0,T) \left( \mathbb{M}|f|^{p_{0}}\right)^{\frac{1}{p_{0}}}, \quad k=0,1\\
(L_{2}f)^{\#}\leq N_{2}(\alpha,d,p_0) \left( \mathbb{M}|f|^{p_{0}}\right)^{\frac{1}{p_{0}}}.
\end{gathered}
\end{equation}
\end{theorem}
\begin{proof} 
The proof is based on \cite[Theorem 3.1]{kim16timefractionalspde}. By the definition, it suffices to show that
\begin{equation}\label{theorem3.6-a}
\aint_{Q}|L_{k}f - (L_{k}f)_{Q}|^{p_{0}}drdz \leq N_k \mathbb{M} |f|^{p_{0}}(t,x)
\end{equation}
for any $(t,x)\in Q=Q_\delta(s,y)$ and $f\in C_c^{\infty}(\mathbb{R}^{d+1})$. By \eqref{eqn 05.16.1}, we may assume 
$$(s+\delta^{2/\alpha},y)=(0,0).$$
This implies that it suffices to consider $Q=Q_{\delta} = Q_\delta(-\delta^{2/\alpha},0)$. Take a function $\zeta$ in $\Ccinf(\R^{d})$ such that $\zeta=1$ on $B_{\delta d}$ and $\zeta=0$ outside of $B_{3\delta d/2}$. Also, choose a function $\eta$ in $\Ccinf(\R)$ such that $\eta=1$ on $[-(3\delta/2)^{2/\alpha},(3\delta/2)^{2/\alpha}]$ and $\eta=0$ outside of $[-(2\delta)^{2/\alpha},(2\delta)^{2/\alpha}]$. For $f$ in $C_{c}^{\infty}(\R^{d+1})$, let $f_{1}=f\zeta$, $f_{2}=f(1-\zeta)\eta$, and $f_{3}=f(1-\zeta)(1-\eta)$. Since $L_k$ are linear, for any real number $c$, 
$$|L_{k}f(s,y)-c|\leq |L_{k}f_{1}(s,y)|+|L_{k}f_{2}(s,y)|+|L_{k}f_{3}(s,y)-c|.
$$
By \eqref{3.20},
\begin{equation*}
\begin{aligned}
&\aint_{Q_{\delta}}|L_{k}f-(L_{k}f)_{Q_{R}}|^{p_{0}}dyds \\
&\quad\leq 2^{p_{0}}\aint_{Q_{\delta}}|L_{k}f-c|^{p_{0}}dyds
\\
&\quad\leq 2^{2p_{0}}\aint_{Q_{\delta}}|L_{k}f_{1}|^{p_{0}}dyds+2^{2p_{0}}\aint_{Q_{\delta}}|L_{k}f_{2}|^{p_{0}}dyds+2^{2p_{0}}\aint_{Q_{\delta}}|L_{k}f_{3}-c|^{p_{0}}dyds.
\end{aligned}
\end{equation*}
$L_{k}f_{1}$ and $L_{k}f_{2}$ can be handled by Lemma \ref{2_2_bound_1}, Lemma \ref{2_2_bound_2}, and Lemma \ref{lem 05.17.1}. It remains to control the last term. Take $c=(L_{k}f_{3})_{Q_\delta}$. Choose $\xi\in\Ccinf(\R)$ such that $0\leq\xi\leq1$, $\xi=1$ if $s\leq \delta^{2/\alpha}$ and $\xi=0$ if $s\geq (3\delta/2)^{2/\alpha}$. By Lemma \ref{2_2_bound_3},
\begin{equation*}
\aint_{Q_{\delta}}\aint_{Q_{\delta}}|L_{k}f_{3}\xi(s,y)-L_{k}f_{3}\xi(r,z)|^{p_{0}}dzdrdyds \leq N_{k} \mathbb{M}|f\xi|^{p_{0}}(t,x).
\end{equation*}
Since $L_{k}f_{3}(s,y)=L_{k}f_{3}\xi(s,y)$ for $s\leq0$, and $|f\xi|\leq|f|$, we have \eqref{theorem3.6-a}. The theorem is proved.
\end{proof}

For $(t,x)$, and $(s,y)$ in $\R^{d+1}$, define a nonnegative symmetric funtion
\begin{equation*}
d_{\alpha}((t,x),(s,y))=|t-s|^{\frac{\alpha}{2}}+|x-y|.
\end{equation*}
Then for any $(t,x),(s,y)$, and $(r,z)$ in $\R^{d+1}$, it follows that
\begin{equation*}
d_{\alpha}((t,x),(r,z))\leq  d_{\alpha}((t,x),(s,y))+d_{\alpha}((s,y),(r,z) )
\end{equation*}
since $\alpha\in(0,2)$. Also, it holds that $d_{\alpha}((t,x),(s,y))=0$ if and only if $(t,x)=(s,y)$. Therefore, the map $d_{\alpha}$ on $\R^{d+1}\times\R^{d+1}$ is a metric. Moreover, the ball 
\begin{equation*}
B^{\alpha}_{r}(t,x)=\{(s,y)|d_{\alpha}((t,x),(s,y))<r\}
\end{equation*}
satisfies the following doubling condition.
\begin{equation}\label{eqn 05.06.4}
|B^{\alpha}_{2r}(t,x)|=2^{\frac{2}{\alpha}+d}|B^{\alpha}_{r}(t,x)|.
\end{equation}
Therefore, there exists a filtration of partitions $\mathbb{C}_{n},n\in\bZ$ of $\R^{d+1}$ satisfying the following.
\begin{enumerate}[(i)]
\item For each $n\in\bZ$, 
$$ \left|\R^{d+1}\setminus\bigcup_{\tilde{Q}_{n}\in\mathbb{C}_{n}} \tilde{Q}_{n}\right|=0.$$
\item There exists a constant $\ep\in(0,1)$ depending only on $\alpha,d$ such that $\text{diam}(\tilde{Q}_{n})\leq N_{0}\ep^{n}$ for any $\tilde{Q}_{n}\in\mathbb{C}_{n}$.
\item For any $m\leq n$, and $\tilde{Q}_{n}\in\mathbb{C}_{n}$, there exists a unique $\tilde{Q}_{m}\in\mathbb{C}_{m}$ such that $\tilde{Q}_{n}\subset\tilde{Q}_{m}$.
\item Each $\tilde{Q}_{n}\in\mathbb{C}_{n}$ contains some ball $B^{\alpha}_{\ep_{0}\ep^{n}}(t,x)$, where the constant $\varepsilon_{0}>0$ depends only on $\alpha,d$
\end{enumerate}
(e.g., \cite[Theorem 2.1]{dong18Apweights}).
By using this one can define a dyadic sharp function of a measurable function $h$ as follows.
\begin{equation*}
\begin{aligned}
h^{\#}_{\text{dy}}(t,x)=\sup_{n\in\bZ}\aint_{\tilde{Q}_{n}}|f(s,y)-(f)_{\tilde{Q}_{n}}|dsdy.
\end{aligned}
\end{equation*}
Also $Q_{\delta}(t,x)$ and $B^{\alpha}_{\delta}(t,x)$ have the following relations.
\begin{equation}\label{eqn 05.06.5}
\begin{gathered}
B^{\alpha}_{\delta}(t,x)\subset Q_{2\delta}(t,x),\quad Q_{\delta}(t,x)\subset  B^{\alpha}_{(2^{-\alpha/2}+\sqrt{d}/2)\delta}(t,x), \quad (t,x)\in\R^{d+1}.
\\
\end{gathered}
\end{equation}
Moreover, using the propery of $\mathbb{C}_{n}$, the doubling condition \eqref{eqn 05.06.4} of $d_{\alpha}$, and \eqref{eqn 05.06.5}, we get
\begin{equation}\label{eqn 05.06.3}
h^{\#}_{\text{dy}}\leq N h^{\#}(t,x),\quad (t,x)\in\R^{d+1}\quad (a.e.),
\end{equation}
where the constant $N$ depends only on $\alpha,d,\varepsilon_{0}$.

\vspace{4mm}

\textbf{Proof of Theorem \ref{thm 05.06.3}} 

Let $f\in C_c^\infty(\bR^{d+1})$ be given.  For each $n\in \bN$, take smooth $\phi_n\in C^\infty(\bR)$ so that $0\leq\phi_n\leq 1$, $\phi_n = 1$ for $t\leq T$, and $\phi_n = 0$ for $t\geq T+1/n$. By Remark \ref{rmk 09.27.21:52} there exist $p_{1}\in(1,p)$, and $q_{1}\in(1,q)$ such that $w_{1}\in A_{p_{1}}$, and $w_{2}\in A_{q_{1}}$. Choose $p_0\in(1,\infty)$ such that 
\begin{equation*}
p_{1}<p/p_{0}<p, \quad  q_{1}<q/p_{0}<q.
\end{equation*} 
Then, it follows that $w_{1}\in A_{p/p_{0}}$, and $w_{2}\in A_{q/p_{0}}$. By a version of the Fefferman-Stein theorem (\cite[Corollary 2.7]{dong18Apweights} with $d_{\alpha}$,) and \eqref{eqn 05.06.3}, we have
\begin{equation*}
\begin{aligned}
\|L_{k}f\phi_n\|_{\tilde{\bL}(q,p,w_{2},w_{1})} &\leq N \|(L_{k}f\phi_n)^{\#}_{\text{dy}}\|_{\tilde{\bL}(q,p,w_{2},w_{1})}\leq N \|(L_{k}f\phi_n)^{\#}\|_{\tilde{\bL}(q,p,w_{2},w_{1})}.
\end{aligned}
\end{equation*}
By using Theorem \ref{thm 05.06.2}, and a version of the Hardy-Littlewood theorem(e.g. \cite[Corollary 2.6]{dong18Apweights} with Euclidean mertic on $\R^{d+1}$),
\begin{equation*}
\begin{aligned}
 \|(L_{k}f\phi_{n})^{\#}\|_{\tilde{\bL}(q,p,w_{2},w_{1})}
&\leq N_k \|(\mathbb{M}|f\phi_{n}|^{p_{0}})^{1/p_{0}}\|_{\tilde{\bL}(q,p,w_{2},w_{1})}
\\
&\leq N_k \||f\phi_{n}|^{p_{0}}\|_{\tilde{\bL}(q/p_{0},p/p_{0},w_{2},w_{1})}^{1/p_{0}}
\\
&= N_k \|f\phi_{n}\|_{\tilde{\bL}(q,p,w_{2},w_{1})},
\end{aligned}
\end{equation*}
where $N_k=N_k(\alpha,d,p,q, [w_{1}]_{p},[w_{2}]_{q},T)$ ($k=0,1$), and $N_2=N_2(\alpha,d,p,q, [w_{1}]_{p},[w_{2}]_{q})$.

Since 
$$\|L_{k}f\|_{\tilde{\bL}(q,p,w_{2},w_{1},T)}=\|L_{k}(f\phi_{n})\|_{\tilde{\bL}(q,p,w_{2},w_{1},T)}\leq \|L_{k}(f\phi_{n})\|_{\tilde{\bL}(q,p,w_{2},w_{1})},$$ 
and 
$$\lim_{n\to\infty}\|f\phi_{n}\|_{\tilde{\bL}(q,p,w_{2},w_{1},T+1/n)}=\|f\|_{\tilde{\bL}(q,p,w_{2},w_{1},T)},$$ 
it holds that
\begin{equation*}
\begin{aligned}
\|L_{k}f\|_{\tilde{\bL}(q,p,w_{2},w_{1},T)}\leq N_k \|f\|_{\tilde{\bL}(q,p,w_{2},w_{1},T)}.
\end{aligned}
\end{equation*}
The theorem is proved.
\qed

\mysection{Proof of Theorem \ref{theorem 5.1}}

First we prove  the a prior estimate and the uniqueness.
Suppose that $u\in\bH^{\alpha,2}_{q,p,0}(w_{2},w_{1},T)$ is a solution to equation \eqref{eqn 05.09.1}. Take $u_{n}\in\Ccinf((0,\infty)\times\R^{d})$ which converges to $u$ in $\bH^{\alpha,2}_{q,p}(w_{2},w_{1},T)$. Let $f_{n}:=\partial^{\alpha}_{t}u_{n}-\Delta u_{n}$. Then, by Lemma \ref{representation_whole_space} \eqref{representation}, 
\begin{equation} \label{during_proof_5}
u_n(t,x)=\int_{0}^{t}\int_{\R^{d}}q(t-s,x-y)f_n(s,y)dyds, \quad (t,x)\in(0,T)\times\bR^d.
\end{equation}
Obviously,  if $0<s<t<T$  then  $t-s\in(0,T)$. Also note that $f_n(t,\cdot) = 0$ for all small $t>0$. Thus, by extending  $f_n(t)=0$ for $t<0$,  we have 
$$u_{n}=L^T_{0}f_{n}, \quad D_{i}u_{n}=L^{T,i}_{1}f_{n}, \quad D_{ij}u_{n}=L_{2}^{ij} f_{n}, \quad t\in (0,T).
$$
 Observe that $f_n\in \tilde{\bL}(q,p,w_{2},w_{1},T)$ since $\partial_t^\alpha u_n,\Delta u_n \in\bL_{q,p}(w_{2},w_{1},T)$ and $f_n = 0$ for $t\leq0$. 
Thus, by Theorem \ref{thm 05.06.3} and Remark \ref{remark 10.24.1},
\begin{equation*}
\begin{aligned}
\|u_{n}\|_{\bH^{0,2}_{q,p}(w_{2},w_{1},T)}&\leq \|L_{0}f_{n}\|_{\tilde{\bL}_{(q,p,w_{2},w_{1},T)}}+\|L_{1}f_{n}\|_{\tilde{\bL}_{(q,p,w_{2},w_{1},T)}}
+\|L_{2}f_{n}\|_{\tilde{\bL}_{(q,p,w_{2},w_{1},T)}}
\\
&\leq N \|f_{n}\|_{\tilde{\bL}_{(q,p,w_{2},w_{1},T)}}= N \|f_{n}\|_{\bL_{q,p}(w_{2},w_{1},T)}.
\end{aligned}
\end{equation*}
Since $\partial^{\alpha}_{t}u_{n}=\Delta u_{n}+f_{n}$, we have
\begin{equation}
    \label{eqn 10.20.5}
\|u_{n}\|_{\bH^{0,2}_{q,p}(w_{2},w_{1},T)}+\|\partial^{\alpha}_{t}u_{n}\|_{\bL_{q,p}(w_{2},w_{1},T)}\leq N\|f_{n}\|_{\bL_{q,p}(w_{2},w_{1},T)}.
\end{equation}
Letting $n\to\infty$, we obtain estimate \eqref{eqn 05.08.1}, and  the uniqueness also follows. 

Next, we  prove the existence. Let $f\in\bL_{q,p}(w_{2},w_{1},T)$, and take $f_{n}\in\Ccinf((0,\infty)\times\R^{d})$ which converges to $f$ in $\bL_{q,p}(w_{2},w_{1},T)$. For each $n$ define $u_{n}$ as
\begin{equation*}
u_{n}(t,x):=\int_{0}^{t}\int_{\R^{d}}q(t-s,x-y)f_{n}(s,y)dyds.
\end{equation*}
By Lemma \ref{representation_whole_space} and \eqref{eqn 10.20.5}, $u_n \in \bH^{\alpha,2}_{q,p,0}(w_{2},w_{1},T)$ and it satisfies
$$
\partial^{\alpha}_t u_n=\Delta u_n +f_n, \quad t>0.
$$
Also, by \eqref{eqn 10.20.5} again, we conclude that  $u_{n}$ is a Cauchy in $\bH^{2,\alpha}_{q,p,0}(w_{2},w_{1},T)$.  Finally, taking $n\to \infty$, we obtain a solution to equation   \eqref{eqn 05.09.1} in the space $\bH^{2,\alpha}_{q,p,0}(w_{2},w_{1},T)$. The theorem is proved.

\end{document}